\documentclass[pdflatex,sn-mathphys-num]{sn-jnl}

\usepackage{amsmath,amssymb,amsfonts}
\usepackage{mathrsfs}
\usepackage{bm}
\usepackage{algorithm}
\usepackage{algpseudocode}
\usepackage{booktabs}
\usepackage{multirow}
\usepackage{rotating}
\usepackage{graphicx}
\usepackage{float}
\usepackage{subeqnarray}
\usepackage{hyperref}
\usepackage{caption}
\usepackage{siunitx}
\usepackage{xcolor}

\theoremstyle{thmstyleone}%
\newtheorem{theorem}{Theorem}[section]

\newtheorem{proposition}[theorem]{Proposition}
\newtheorem{corollary}[theorem]{Corollary}

\theoremstyle{thmstyletwo}%
\newtheorem{example}{Example}[section]

\theoremstyle{thmstylethree}%

\newtheorem{assumption}{Assumption}[section]

\renewcommand{\theequation}{\thesection.\arabic{equation}}

\def\diag{{\rm diag}\,}
\def\ba{\begin{array}}
\def\ea{\end{array}}
\def\beq{\begin{equation}}
\def\eeq{\end{equation}}
\def\bea{\begin{eqnarray}}
\def\eea{\end{eqnarray}}
\def\beann{\begin{eqnarray*}}
\def\eeann{\end{eqnarray*}}
\newcommand{\reff}[1]{\mbox{(\ref{#1})}}

\def\sgn#1{\text{sign}\left(#1\right)}

\makeatletter
\renewcommand{\section}{\@startsection {section}{1}{\z@}
  {-8pt plus -2pt minus -1pt}
  {4pt plus 2pt minus 1pt}
  {\normalfont\Large\bfseries}}

\renewcommand{\subsection}{\@startsection {subsection}{2}{\z@}
  {-6pt plus -2pt minus -1pt}
  {3pt plus 1pt minus 1pt}
  {\normalfont\large\bfseries}}

\renewcommand{\subsubsection}{\@startsection {subsubsection}{3}{\z@}
  {-4pt plus -1pt minus -1pt}
  {2pt plus 1pt minus 1pt}
  {\normalfont\normalsize\bfseries}}
\makeatother

\date{\today}

\begin{document}

\title[ BPMs for SLCC]{  A Relaxation and Rectification (ReCR) Framework for Systems with Linear and Complementary Constraints: Theoretical Foundation, Algorithms and
Numerical Experiments}

\author[1]{\fnm{Wissam} \sur{AlAli}}\email{walali@uh.edu}
\author[1] {\fnm{Xin} \sur{Jiang}}\email{xinjiang@uh.edu}
\author*[1]{\fnm{Jiming} \sur{Peng}}\email{jopeng@uh.edu}

\affil*[1]{\orgdiv{Department of Industrial and Systems Engineering},
  \orgname{University of Houston},
  \orgaddress{\street{4800 Calhoun Rd}, \city{Houston}, \postcode{77204}, \state{TX}, \country{USA}}}

\abstract{
Systems defined by linear and complementarity constraints (SLCCs) arise frequently in engineering, economics, and other related fields. They also appear in the optimality conditions of many challenging optimization models, such as bilinear optimization and linearly constrained quadratic optimization. It is known that finding a feasible solution to an SLCC is NP-hard in general. In this paper, we study the feasibility problem for a given SLCC: either find a feasible solution or determine that the system is infeasible. To this end, we introduce a universal relaxation theory (URT), which reformulates SLCC feasibility as an equivalent bilinear optimization problem with linear constraints in a lifted space. We then analyze the resulting bilinear model and derive necessary and sufficient optimality conditions for its global solutions. Based on these theoretical insights, we introduce a relaxation-rectification (ReCR) framework for finding a feasible solution to a given SLCC instance or certifying infeasibility. We develop several ReCR methods that differ in their working spaces and subproblem formulations and analyze their convergence properties. We also develop a numerical procedure for obtaining an infeasibility certificate when the ReCR methods do not find a feasible solution.

	We conduct numerical experiments to evaluate the reliability, robustness, and scalability of the proposed ReCR methods and compare them with existing SLCC solvers. On the tested small- and large-scale LCP instances from the literature, the proposed ReCR methods typically find feasible solutions in a few iterations. We also extend the benchmark with more challenging medium-scale SLCC instances, on which the proposed hybrid ReCR (H-ReCR) method exhibits promising performance.
}


\maketitle

\section{Introduction}
   In this paper, we consider   the following   system   with linear and complementarity constraints (SLCC): \vspace{-0.05in}
 \begin{subequations} \label{mod:SLCC}
 \begin{align}    Mx+Qz-r \ge 0, &  \quad                Ax+Bz = b;   \label{LC-1} \\
           x^T (Mx+Qz-r)=0, &\quad  x,z\ge 0 \label{CC}
  \end{align}
\end{subequations}
  with variables $x\in \mathbb{R}^n$ and $z\in \mathbb{R}^m$. Here $M\in \mathbb{R}^{n\times n}$, $Q\in \mathbb{R}^{n\times m}$, $A\in \mathbb{R}^{p\times n}$, $B\in \mathbb{R}^{p\times m}$ are matrices, and $r\in \mathbb{R}^n$, $b\in\mathbb{R}^p$ are vectors.     SLCC~(\ref{mod:SLCC})
  arises in various  engineering applications \cite{Cottle1992LCP,Facchinei2002VICP,Judice2014LOCC}.
    It appears in the optimality conditions of several important  classes of challenging optimization problems such as bilinear optimization and bilevel optimization,   linearly constrained quadratic optimization (LCQO) and optimization problems with cardinality constraints \cite{Ferris1997MPEC,Judice2014LOCC,Luo1997MPEC,Mitchell2019LPCC}. As such, SLCCs can also serve as building blocks in the development of solution methods for these models. It has been shown that finding a feasible solution to an SLCC is NP-hard in general \cite{Judice2014LOCC, Pardalos1988}.

  An SLCC can be reformulated as  a specific LCQO, so methods such as
  alternative direction methods (ADM)  \cite{Boyd2011ADM,Konno1976MP,Gomez1982SLP}, interior-point methods \cite{Anjos2012ConicOpt,Ye1998MP} and the D.C.\ programming methods \cite{Horst1999DC,Thi2011DC-LCP} can be used to solve the reformulated LCQO.
  However, due to the nonconvexity of the objective function, these methods generally guarantee only stationarity for the reformulated LCQO; recovery of a solution to the original SLCC typically requires additional assumptions on the data matrices. For additional studies on stationary-point methods for optimization problems with linear and complementarity constraints, we refer to    \cite{Leyffer2004SIOPT,Luo1997MPEC,Pang1999COA, Scholtes2000MOR,JYe1999SIAMOpt,Ye1993MOR} and the references therein. 

  Another popular approach for SLCCs is to reformulate the complementarity conditions as a system of nonsmooth equations and then apply algorithms for solving such systems. These include Lemke's pivotal method \cite{Lemke1964LCP}, the iterative projection method \cite{He1996JCM-LCP}, and path-following methods \cite{Ferris1999Path, Ferris1999MP}. These methods can be effective for feasible SLCCs when the data matrices satisfy suitable assumptions.  Alternatively, an SLCC can be cast as a nonconvex optimization problem in which the objective function is defined by a merit function associated with the equation system \cite{Dirkse2005MPEC, Facchinei2002VICP, Ferris1999MP, Fukushima1996}. Most conventional optimization methods generally guarantee only convergence to a stationary point of the merit-function problem \cite{Boyd2011ADM, Judice2014LOCC},while convergence to a feasible solution of the original SLCC requires additional assumptions on the data matrices \cite{Dirkse2005MPEC, Facchinei2002VICP, Ferris1999MP, Fukushima1996}. However, it is unclear how to extend such results to generic SLCCs. In principle, the complementarity conditions  in  SLCCs can also be modeled using  binary variables \cite{Judice2014LOCC,Mitchell2019LPCC}, so global optimization methods for binary optimization, such as branch-and-bound methods, can be applied to solve SLCCs \cite{Wachter2006IPOPT,Sahinidis1996Baron}. As pointed out in \cite{Mitchell2019LPCC}, such global optimization methods are often not scalable for large-scale instances.

In the past two decades, the field of optimization has gone through exciting developments. Various approaches based on conic optimization relaxations have been proposed to find high-quality approximate solutions or tight bounds for challenging optimization problems \cite{Anjos2012ConicOpt,Luo2010QCQO}.
Several effective approaches for convex optimization, such as ADMs \cite{Boyd2011ADM} and first-order methods (FOMs) \cite{Beck2017FOM}, have been developed. Under appropriate assumptions, such methods have well-established convergence guarantees for convex optimization problems \cite{Beck2017FOM,Boyd2011ADM}. This motivates the question of whether convex relaxation models can be combined with efficient solution methods to find a feasible solution of an SLCC or certify its infeasibility.

The main difficulty in addressing an SLCC lies in handling its complementarity conditions. To reduce this difficulty, one may relax the complementarity conditions in (\ref{mod:SLCC}) to obtain a linear feasibility problem that can be solved efficiently. One benefit of this relaxation is the flexibility to choose an objective function, which allows us to extend the working space by introducing parameters in the objective function. This leads to the following question: does there exist a linear objective function such that the resulting linear optimization problem has an optimal solution from which one can recover a feasible solution of the original SLCC \reff{mod:SLCC}, or certify that no such solution exists?  An affirmative answer to this question would provide a theoretical foundation for a framework that seeks such a desirable objective function. The first objective in this paper is to develop a theoretical framework, called universal relaxation theory (URT), to establish the existence of such a desirable linear objective for SLCC \reff{mod:SLCC}.

Second, we consider how to identify such a desirable linear objective function, which leads to an equivalent bilinear optimization problem with linear constraints. Various approaches have been proposed for bilinear optimization \cite{Beck2017FOM,Gorski2007BICO, Konno1976MP,Horst1999DC,Ye1998MP}. These methods typically generate sequences converging to stationary points or partial optima of the underlying bilinear problem; global optimality generally requires additional assumptions on the data matrices. Recall that the necessary and sufficient optimality conditions for global solutions play an important role in the design and analysis of algorithms for global optimization.  The second objective of this paper is to explore the resulting bilinear optimization problem and derive necessary and sufficient conditions at its global optimal solution. Then, based on these theoretical insights, we introduce a relaxation-rectification framework for solving the associated bilinear reformulation. The framework, called relaxation-rectification (ReCR), works as follows. At each iteration, ReCR  solves a relaxed LO or a relaxed convex QO. Based on the optimal solution of the current relaxation model, we then rectify the objective function and solve a new relaxed optimization problem with the rectified objective. This process is repeated until a stopping criterion is met.

We introduce three ReCR variants. These variants differ in their working spaces, subproblem formulations, and objective-update schemes; we call them primal-dual ReCR (PD-ReCR), primal ReCR (P-ReCR), and rectified quadratic relaxation (ReQR). In particular, P-ReCR and ReQR work in the original variable space but solve different relaxation models (LO and QO). Since the feasibility/infeasibility detection of an SLCC requires a highly accurate solution to the underlying bilinear model and  P-ReCR is a first-order method that may have difficulty in obtaining such a highly accurate solution to the reformulated bilinear model, we also introduce a hybrid variant (H-ReCR) that combines P-ReCR with  ReQR to find a highly accurate solution to the bilinear model. We analyze the convergence properties of these ReCR approaches under stated assumptions.

We conduct numerical experiments on various SLCC instances and compare the proposed ReCR approaches with several existing solvers, including the PATH solver \cite{Ferris1999Path}, Lemke’s pivoting method \cite{Lemke1964LCP} and a path-following method based on the Fischer–Burmeister reformulation \cite{Fischer1995}. Specifically, we assess the reliability and robustness of the proposed ReCR approaches using a benchmark of ten small-scale LCP instances from the literature with different matrix structures and right-hand-side vectors. We also test scalability on several large-scale LCP instances with structured sparse matrices. The results show that the proposed ReCR approaches find feasible solutions for all tested LCP instances from the literature in few iterations. Motivated by these results, we combine randomization with deterministic construction rules to generate synthetic SLCC instances, including infeasible instances with feasible relaxations and instances with nonconvex solution sets. These synthetic instances extend existing SLCC benchmarks and support the further development of solution techniques. The results show that the proposed ReCR approaches find feasible solutions for the feasible SLCC instances in the extended library and provide infeasibility certificates for the infeasible instances.

The paper is organized as follows. In Section \ref{Sect:URT}, we introduce the universal relaxation theory for LCPs and use it to reformulate a given LCP as an equivalent bilinear optimization problem. We also analyze the resulting bilinear model and derive necessary and sufficient optimality conditions for its global solutions. In Section~\ref{Sec:ReCR-LCP}, we introduce three ReCR variants for generic LCPs, primal--dual ReCR, primal ReCR, and rectified quadratic relaxation (ReQR), and investigate their convergence properties. We also introduce a procedure for checking infeasibility when ReCR does not find a feasible LCP solution. In Section~\ref{sec:ReCR-SLCC}, we extend P-ReCR and ReQR to generic SLCCs. In Section~\ref{sec:ReCR-SLCC}, we evaluate the performance of the proposed ReCR approaches on SLCC instances from the literature and on synthetic instances generated in this work. Finally, Section \ref{sec:conclusion} concludes the paper and discusses future research directions.

\section{New Theoretical Results for Linear Complementary Problems}\label{Sect:URT}
In this section, we introduce some new theoretical results that form the theoretical foundation for the new ReCR framework for LCPs.
The section consists of two subsections. In the first subsection, we introduce the \textit{universal relaxation theory} (URT) for LCP, which shows that we can address the feasibility issue of an LCP by solving an equivalent bilinear optimization problem.  
 In the second subsection, we recast a given LCP as an equivalent bilinear optimization model and explore the constructed bilinear model to derive necessary and sufficient conditions for its global optimal solution.

\subsection{Universal Relaxation Theory for LCP}

For simplicity, we first consider the following LCP:
\begin{eqnarray} \label{mod:LCP}
 Mx\ge r, \qquad x\ge 0, \qquad x^T(Mx-r)=0.
\end{eqnarray}
Solving model~(\ref{mod:LCP}) is NP-hard \cite{Cottle1992LCP,Judice2014LOCC}.  We observe that the main difficulty in solving an LCP lies in the complementary constraints.  A common strategy in handling non-convex constraints is to relax these hard constraints as convex ones. On the other hand, we have the freedom to choose the objective function in the following LO:
\begin{subequations}\label{mod:PLO-LCP}
 \begin{eqnarray}      \min &&\quad c^Tx \\
         \text{s.t.} && \quad  Mx \ge r, \quad x\ge 0.
  \end{eqnarray}
\end{subequations}
Two natural questions arise:
  \begin{itemize}
  \item[]{\bf Q.1:} Does there exist a vector $c$ such that the optimal solution of model~(\ref{mod:PLO-LCP}) is a feasible solution to the original LCP?
  \item[]{\bf Q.2:} If the answer to Q.1 is ``yes'', how can we find such a desirable objective function?
    \end{itemize}
  In what follows, we try to address these two questions.
   First, we will develop a new theory, called universal relaxation theory, which shows that if the underlying LCP is feasible, then there does exist such  a desirable objective function.
     To establish the URT for model (\ref{mod:PLO-LCP}),  we consider its dual  as follows:
 \begin{subequations}\label{mod:DLO-LCP}
 \begin{eqnarray}      \max &&\quad r^Ty \\
         \text{s.t.} && \quad  M^Ty \le c, \quad y\ge 0.
  \end{eqnarray}
  \end{subequations}
Let $\mathcal{X}_\mathrm{LCP}$ denote the solution set of problem \reff{mod:LCP}. We assume that problem \ref{mod:PLO-LCP} is feasible; otherwise we can conclude that
$\mathcal{X}_\mathrm{LCP}$ is empty.  We have
\begin{theorem}\label{thm:URT} Suppose that problem \ref{mod:PLO-LCP} is feasible.
$\mathcal{X}_\mathrm{LCP}$ is nonempty if and only if there exists some vector $c$ under which there exists a primal--dual pair $(x^c,y^c)$ of the optimal solutions to  problems (\ref{mod:PLO-LCP}) and (\ref{mod:DLO-LCP}) respectively such that the following inequalities hold
  \begin{eqnarray} x^c_i\le  y^c_i, \quad \forall i\in \mathcal{I}=\{1,2,\cdots,n\}. \label{Opt-Cons}\end{eqnarray}
\end{theorem}
\begin{proof}  The sufficiency of the theorem follows directly from the duality theory for LO. It suffices to prove the necessity part of the theorem. Since $\mathcal{X}_\mathrm{LCP}$ is nonempty,  we can choose an arbitrary  $x^*\in \mathcal{X}_\mathrm{LCP}$  and define
\begin{eqnarray}\label{C-Update1} c=M^T x^*. \end{eqnarray}
Let $y^*=x^*$ be the vector of the dual variables in (\ref{mod:DLO-LCP}).
               It is easy to see that $y^*$ is a feasible solution to \reff{mod:DLO-LCP} and $c^Tx^*=r^Ty^*$. Therefore,   $x^*$ and $y^*$ are optimal solutions to (\ref{mod:PLO-LCP}) and  (\ref{mod:DLO-LCP}) respectively.  This completes the proof of Theorem~\ref{thm:URT}. \end{proof}

                We remark that Theorem~\ref{thm:URT} not only establishes an intrinsic relationship   between  the set $\mathcal{X}_\mathrm{LCP}$ and the primal--dual pair of optimal solutions to problems (\ref{mod:PLO-LCP}) and (\ref{mod:DLO-LCP}),
               but also provides theoretical insight to help design effective algorithms to identify a desirable vector $c$ under which we can construct a feasible point in $\mathcal{X}_\mathrm{LCP}$ from  the optimal solution of problem~(\ref{mod:PLO-LCP}), or determine it is impossible to do so. We next present a technical result, which is a generalization of  Theorem~\ref{thm:URT}.
     \begin{corollary}\label{cor:URT}
$\mathcal{X}_\mathrm{LCP}$ is nonempty if and only if there exists some vector $c$ under which there exists a primal--dual pair $(x^c,y^c)$ of the optimal solutions to  problems (\ref{mod:PLO-LCP}) and (\ref{mod:DLO-LCP}) respectively such that the following inequalities hold
  \begin{eqnarray}\label{Opt-Cons2} x^c_i\le \tau y^c_i, \quad \forall i\in \mathcal{I}=\{1,2,\cdots,n\}\end{eqnarray}
  for some $\tau>0$.
\end{corollary}
  We remark that such a generalization can facilitate the design of the learning scheme. This is because the ideal case $\tau=1$ in Theorem \ref{thm:URT} holds only when $x^*$ is known a prior. Unfortunately,   $x^*$ is usually unknown  in advance.
  Instead, we know only some vector $c$ and the optimal solutions to model \reff{mod:PLO-LCP} and its dual denoted by $x^c$ and $y^c$ respectively.
    On the other hand, we observe that the optimal solution to model \reff{mod:PLO-LCP} is invariant under vector scaling, i.e.,  if we scale the vector $c$ to  $\tau c$ for arbitrary $\tau>0$, then the optimal solution to the scaled model \reff{mod:PLO-LCP} remains the same while the optimal solution  to the scaled dual problem will be changed to $\tau y^*$. This  shows that if the inequalities \reff{Opt-Cons2} hold, then we can scale the vector $c$ to $\tau c$ so that the inequalities \reff{Opt-Cons} hold at the optimal solutions to the scaled model and its scaled dual.

Alternatively, we can reformulate a given LCP as the following quadratic optimization problem:
\begin{subequations}\label{mod:QO-LCP}
 \begin{eqnarray}      \min &&\quad f_\mathrm{q}(x)=x^TMx-r^Tx \\
         \text{s.t.} && \quad  Mx \ge r, \quad x\ge 0.
  \end{eqnarray}
  \end{subequations}
The QO model \reff{mod:QO-LCP} can be viewed as a restricted reformulation obtained by choosing the objective vector as $c=M^T x$ in the bilinear optimization model \reff{mod:PDLO}. The great flexibility in the selection of $c$ in model \reff{mod:PLO-LCP} facilitates the design of new ReCR approaches for the original LCP.  Specifically,  the so-called  rectified quadratic relaxation (ReQR) approach amounts to the well-known D.C. program \cite{Horst1999DC} for \reff{mod:QO-LCP}. Another equivalent reformulation casts a given LCP as the following system of non-smooth equations:
  \begin{eqnarray}\label{mod:SEQ-LCP}
       \phi(x)=\min(x, Mx-r)=0.
  \end{eqnarray}

\subsection{ A bilinear optimization model for LCP and its properties}

Based on Theorem~\ref{thm:URT}, we propose to solve the following bilinear optimization problem
\begin{subequations} \label{mod:PDLO}
 \begin{eqnarray}  \min &&\quad  c^Tx - r^Ty \\
    \text{s.t.} &&\quad    Mx \ge r,  \quad  M^Ty \le c; \\
        &&\quad           x\le  y, \quad x,y\ge 0,
  \end{eqnarray}
\end{subequations}
where the decision variables are $c,x,y$. We refer to \reff{mod:PDLO} as the bilinear optimization (BLO) problem. For a fixed vector $c$, \reff{mod:PDLO} reduces to a linear optimization problem in variables $(x,y)$; we denote this fixed-$c$ linear optimization problem by $\mathrm{LO}(c)$. As shown in the following proposition, solving the BLO model \reff{mod:PDLO} can help us either find a feasible solution to the original LCP model \reff{mod:LCP} or determine it is infeasible.
\begin{proposition}\label{Prop:BLO} Suppose that model~\reff{mod:PDLO} is feasible. Let $(c^*, x^*, y^*)$  be the global optimal solution to problem \reff{mod:PDLO}.
If $(c^*)^T x^*-r^Ty^*=0$, then $x^*$ is a solution to model~\reff{mod:LCP}. If $(c^*)^T x^*-r^T y^*>0$, then model~\reff{mod:LCP} is infeasible.
\end{proposition}

Note that due to the bilinearity of the objective function in \reff{mod:PDLO}, it is still challenging to locate its global optimal solution.
As  pointed out in the introduction, though numerous optimization methods such as the successive linear optimization approach  and the D.C. program can be applied to solve \reff{mod:PDLO}, these algorithms can only find a stationary point that may not be the global optimal solution to model \reff{mod:PDLO}.

Next, we explore properties of the linear optimization problem, denoted by $\mathrm{LO}(c)$, obtained from \reff{mod:PDLO} after fixing $c$. The dual of $\mathrm{LO}(c)$ is
\begin{subequations} \label{mod:DDLO}
 \begin{eqnarray}  \max &&\quad   r^Tv-c^T u \\
        \text{s.t.} &&\quad M u- s \ge r,  \quad  M^Tv- s \le c; \\
        &&\quad           u,v,s\ge 0.
  \end{eqnarray}
\end{subequations}
Recall that for a feasible and bounded linear program, there exists a primal--dual pair of optimal solutions that satisfy strict complementarity. For fixed~$c$, let $(x^c, y^c)$ and $(u^c,v^c,s^c)$ denote a strictly complementary primal--dual optimal solution of $\mathrm{LO}(c)$ and its dual (\ref{mod:DDLO}), respectively.  Then the following conditions hold for every $i\in \mathcal{I}$
\begin{subequations} \label{PD-conds}
 \begin{eqnarray}    (c-M^Tv^c+s^c)_ix^c_i=0, && x^c_i+c_i-(M^Tv^c)_i+s^c_i>0; \nonumber\\
                         (Mx^c-r)_i v^c_i=0, &&  v^c_i+(Mx^c)_i-r_i>0; \nonumber\\
          (Mu^c- s^c-r)_iy^c_i=0, &&  y^c_i+ (Mu^c- s^c-r)_i>0; \nonumber\\
          (c-M^Ty^c)_i u^c_i=0, &&   u^c_i+c_i-(M^Ty^c)_i>0; \nonumber\\
               ( y^c-x^c)_i s^c_i=0, &&   y^c_i-x^c_i+s^c_i>0. \nonumber
  \end{eqnarray}
\end{subequations}
From the above optimality conditions, if $x^c_i>0$, then $(M^Tv^c)_i-s^c_i=c_i$, and hence $s_i^c = (M^Tv^c-c)_i$. For indices with $x^c_i=0$, the value of $s_i^c$ may not be uniquely determined by these conditions. In what follows, we select a strictly complementary primal--dual optimal solution for which the slack vector $s^c$ is represented as
\begin{eqnarray} \label{def:sc}
  s^c=\max(M^Tv^c-c,0).
\end{eqnarray}
This selection is used in the subsequent analysis.
Similarly, we can also conclude that $(Mu^c-r)_i=s^c_i$ whenever $x^c_i>0$. Therefore, we derive the following equivalent form for $s^c$:
\begin{eqnarray}\label{def2:sc} s^c=\diag(\max(Mu^c-r,0)) \cdot \sgn{x^c},
           \end{eqnarray}
           where $\sgn{x^c}$ is the sign function defined by
           \[ \sgn{x^c}_i=\left\{\begin{array}{ll}
                                    1 &\quad \mbox{if $x^c_i>0$;}\\
                                    0 &\quad \mbox{otherwise.}
                                    \end{array}\right.\]

Now we state conditions, expressed in terms of the fixed-$c$ LO pair $\mathrm{LO}(c)$--\reff{mod:DDLO}, under which the corresponding solution $x^c$ solves the original LCP.
\begin{theorem}\label{thm:main1} For fixed $c$, suppose that $\mathrm{LO}(c)$ and its dual \reff{mod:DDLO} are feasible.  Let $(x^c, y^c)$ and $(u^c,v^c,s^c)$ be a strictly complementary primal--dual optimal solution of $\mathrm{LO}(c)$ and \reff{mod:DDLO}, respectively.  Then $s^c$ can be set as
\[
  s^c=\max(M^Tv^c-c,0).
\]
Moreover, $x^c\in \mathcal{X}_\mathrm{LCP}$ if and only if at least one of the following two conditions is satisfied:
\begin{itemize}\itemsep-2pt
\item[(i):]  $ c^Tx^c-r^Ty^c=c^Tu^c-r^Tv^c=0$; \quad (ii): $s^c=0$.
\end{itemize}
\end{theorem}
\begin{proof} We need only to prove the conclusion (ii). For this, we first observe that from the optimality of $(u^c,v^c,s^c)$ we obtain
\[ r^Tv^c-c^Tu^c\ge 0, \]
and the equality holds whenever $s^c=0$. This completes the proof of the theorem. \end{proof}

 We remark that to the best our knowledge, this is the very first time that different sufficient and necessary conditions for the global optimal solution to some non-convex bilinear optimization problem are reported in the literature. We also observe that the second condition is derived from the optimal solution of the dual problem \reff{mod:DDLO}.
It is worth mentioning that $s^c$ is similar to the loss function in supervised learning.  However, unlike in supervised learning where the loss function can be computed easily, we need to solve some LO first to compute the loss vector $s^c$. Theorem~\ref{thm:main1} suggests two possible strategies for driving the bilinear model \reff{mod:PDLO} toward an LCP solution: one may directly reduce the bilinear objective $c^Tx-r^Ty$, or for fixed $c$, reduce the loss $e^Ts$ obtained from the dual solution of $\mathrm{LO}(c)$.
A key issue is how to update the vector $c$ using the obtained partial optimal solution to ensure the convergence to the global optimal solution of the original bilinear optimization model.

\section{  Rectified Convex Relaxations for LCP}\label{Sec:ReCR-LCP}

 In this section, we  introduce  three different methods for LCPs and study their convergence properties.

 \subsection{The primal--dual rectified convex relaxation}
 In this subsection, we present a specific variant of ReCR which works as follows. We first solve  problem  \reff{mod:DDLO} with fixed $c$ and obtain its optimal solution  $(u^c, v^c, s^c)$.   Then,   we first try to rectify the vector $c$ via the following scheme:
\begin{eqnarray}\label{PD-ReCR}  c^+=
c+s^c.
\end{eqnarray}
Then we resolve the LO model \reff{mod:DDLO} with $c=c^+$. We can repeat such a process until some stop criteria is met.
 Since at each iteration, we solve some LO problem which involves both problem~(\ref{mod:PLO-LCP}) and its dual, thus we call it primal--dual ReCR or PD-ReCR.
It should  be pointed out that the   rectification scheme $c^+=c+s^c$ is very similar to the RELU activation function widely used in deep learning. The scheme can also be viewed as a backward propagation algorithm based on the optimality condition $s^c=0$ as $(u^c,v^c,0)$ is a feasible solution to the following rectified model.
\begin{subequations} \label{mod:DDLO-c+}
 \begin{eqnarray}  \max &&\quad  r^Tv-(c^+)^Tu   \\
    \text{s.t.} &&\quad    Mu -s\ge r,  \quad  M^Tv-s \le c^+; \\
        &&\quad            u,v,s\ge 0.
  \end{eqnarray}
\end{subequations}
It follows immediately
\begin{theorem}\label{thm-1:PD-ReCR}
Let $(u^+, v^+, s^+)$ be the optimal solution to model \reff{mod:DDLO-c+}. Then it holds
\[  r^T v^+-(c^+)^T u^+\ge 0, \]
and the equality holds if and only if $u^c$ is a feasible solution to the original LCP \reff{mod:LCP}.
 \end{theorem}
Define \[ f_\mathrm{pd}(c, u,v)=r^T v-c^T u.\]
 Let $(u^c, v^c, s^c)$ and $(u^+, v^+, s^+)$ denote the  optimal solutions to models \reff{mod:DDLO} and \reff{mod:DDLO-c+} respectively.
 Based on Theorem~\ref{thm-1:PD-ReCR}, we can claim that  the update from $c$ to $c^+$  is successful if the following inequality
 \begin{eqnarray}\label{inq1:PD-ReCR}  f_\mathrm{pd}(c^+, u^+, v^+) < f_\mathrm{pd}(c, u^c, v^c),\end{eqnarray}
 holds.
 However, as shown in the following theorem, sometimes  inequality \reff{inq1:PD-ReCR} may not hold.
\begin{theorem}\label{thm-2:PD-ReCR}
Let $(u^c, v^c, s^c)$ and $(u^+, v^+, s^+)$ be the exact optimal solutions to models \reff{mod:DDLO} and \reff{mod:DDLO-c+} respectively. Then $u^c=u^+$ if and only if $s^c=s^+$, i.e., $u^c$ is also the optimal solution to the following LO
\begin{subequations} \label{Subprob:ReCR}
 \begin{eqnarray}  \min &&\quad  u^Ts^c   \\
    \mathrm{s.t.} &&\quad    Mu \ge r,  \quad  u\ge 0.
  \end{eqnarray}
\end{subequations}
 \end{theorem}
\begin{proof} We first prove the sufficient part of the theorem. If $u^c$ is the optimal solution to problem \reff{Subprob:ReCR} and let $\bar{v}$ be the optimal solution to its dual, then one can easily verify that $(u^c, v^c+\bar{v}, s^c)$ is the optimal solution to problem \reff{mod:DDLO-c+}. To prove the necessity of the theorem, we first observe that from \reff{def2:sc}, we can conclude that $s^c=s^+$ if $u^c=u^+$.
On the other hand, suppose that $u^c$ is not the optimal solution to problem \reff{Subprob:ReCR}, then there exists $\bar{u}$ satisfying the following relationships:
\[
	\bar{u}^T s^c< (u^c)^T s^c, \quad M \bar{u}\ge r, \quad \bar{u}\ge 0.
\]
Combining the above relationships with the convexity of the objective function in model \reff{mod:DDLO-c+}, we can  conclude that
$u^c\neq u^+$, which contradicts to the assumption $u^c=u^+$. This completes the proof of the theorem.
\end{proof}

Theorem~\ref{thm-2:PD-ReCR} not only presents a scenario where inequality \reff{inq1:PD-ReCR} does not hold, but also indicates that the current iterates are optimal for the fixed-$c$ LO pair $\mathrm{LO}(c)$--\reff{mod:DDLO}. Under such a circumstance, one can easily see that the pure ReCR process  won't be able to improve the quadratic objective function $ u^T (Mu-r)$ or $ x^T(Mx-r)$.

For illustration,  the basic scheme  of PD-ReCR    is described below.
\par\noindent\rule{\textwidth}{0.4pt}
	\noindent{\bf \hspace*{1.5cm} Primal--dual ReCR for LCP }
\par\noindent\rule{\textwidth}{0.4pt}
\begin{itemize} \itemsep-2pt
	\item[]{\bf Input:} \( M, r \), stop criteria \( \epsilon > 0 \), maximum iterations \( K \);
	\item[]{\bf Step 0.} Initialization:
	\begin{itemize}\itemsep-1pt
        \item[]{Step 0.1.}  Set \( k = 0 \), \( c^k = \textbf{1} \);
		\item[] {Step 0.2.} Solve problem (\ref{mod:PLO-LCP}) to obtain \( x^k \);
        \item[]{Step 0.3.} If problem (\ref{mod:PLO-LCP}) is infeasible
        \item[] \qquad  \qquad  \quad  The LCP is infeasible;
        \item[] \qquad \qquad End if
        \item[] {Step 0.4.} Update \( k = k + 1 \);
        \item[] {Step 0.5.} Set $c^k=M^T x^{k-1}$;
     \end{itemize}
		\item[]{\bf Step 1.} Solving process
	\begin{itemize}\itemsep-1pt
		\item[]{ Step 1.1.} Solve problem (\ref{mod:DDLO}) with \( c = c^k \) to obtain \( (u^k,v^k,s^k) \);
		\item[]{ Step 1.2.} Update \( k = k + 1 \) and update \( c^k \) using (\ref{PD-ReCR});
       \item[]{Step 1.3.}  Solve problem (\ref{mod:DDLO}) with \( c = c^k \)  for \( (u^k,v^k,s^k)\);
        \item[]{Step 1.4.} If $\|u^{k}-u^{k-1}\|\ge \epsilon$ then go to Step 1.2;
        \end{itemize}
	\item[]{\bf Step 2.}    Solve $\mathrm{LO}(c^k)$ and let $(x^k,y^k)$ be the optimal solution.  Output $x^* = x^k$ after the feasibility check.
\end{itemize}
\par\noindent\rule{\textwidth}{0.4pt}

 We next study the convergence of the PD-ReCR. We have
\begin{theorem}\label{thm1:PD-ReCR} Suppose that  the LO subproblem \reff{mod:DDLO} at every iteration of PD-ReCR is solved exactly with $\epsilon=0$.   If   the sequence $\{c^k\}$ generated by PD-ReCR is  bounded,       then every accumulation point $(c^*,u^*, v^*)$ of the sequence  is the global optimal solution of problem \reff{mod:PDLO} and  $x^*\in \mathcal{X}_\mathrm{LCP}$.
\end{theorem}
\begin{proof}  We first point out that in PD-ReCR, the sequence $\{c^k\}$ is increasing in a monotonic manner. If it is bounded above,  it must converge to its accumulation point $c^*$.
 Now let consider the following LO
\begin{subequations} \label{mod:DDLO-LCP1}
 \begin{eqnarray}  \quad \max &&\quad   r^Tv-(c^*)^T u \\
        s.t. && M u\ge r,  \quad  M^Tv \le c^*; \\
        &&\quad           u,v\ge 0.
  \end{eqnarray}
\end{subequations}
Let $(u^*_{lo},v^*_{lo})$ denote the optimal solution of the above model and $(x^*, y^*)$ the optimal solution to model~\reff{mod:PDLO} with $c=c^*$.
From the duality theory of LO we can conclude that
\[ 0=r^Tv^*_{lo}-(c^*)^Tu^*_{lo}\le r^T v^*-(c^*)^Tu^* \le (c^*)^Tx^*-r^Ty^* \le 0,\]
where the first inequality follows from the fact that model~(\ref{mod:DDLO}) is a relaxation of model~(\ref{mod:DDLO-LCP1}). This completes the proof of the theorem.
\end{proof}

Theorem~\ref{thm1:PD-ReCR} shows that PD-ReCR can find a feasible solution  to   LCP with arbitrary  matrix $M$ under the assumption that the generated sequence $\{ c^k\}$ is bounded.
Next we investigate the convergence behavior of PD-ReCR   when $\{c^k\}$ is unbounded. Let $\bar{c}^k = {c^k}/{\|c^k\|}$. We consider the following scaled  model:
\begin{subequations} \label{mod:scaled-DDLO}
 \begin{eqnarray}  \max &&\quad   \frac{r^Tv}{\|c^k\|} -(\bar{c}^k)^T u \\
        \text{s.t.} &&\quad  M u- s \ge r,  \quad  \frac{M^Tv- s}{\|c^k\|} \le \bar{c}^k; \\
        &&\quad           u,v,s\ge 0.
  \end{eqnarray}
\end{subequations}
Let consider an accumulation point of the sequence $\{\bar{c}^k\}$. W.l.o.g., we assume that
\[ \bar{c}=\lim_{k\to \infty} \bar{c}^k.\]
Since $\{c^k\}$ is increasing, thus we can conclude that $\bar{c}\ge 0$. Let $\bar{v}=\frac{v}{\|c^k\|}$. Using  asymptotic analysis, we can consider the following LO:
\begin{subequations} \label{mod:scaled-DDLO1}
 \begin{eqnarray}
	\max &&\quad r^T\bar{v}-\bar{c}^T u \\
	\text{s.t.} &&\quad M u- s \ge r,  \quad  M^T\bar{v} \le \bar{c}; \\
	&&\quad u,\bar{v},s\ge 0.
  \end{eqnarray}
\end{subequations}
Let $\bar{x}$ be the optimal solution to model \reff{mod:PLO-LCP} with $c=\bar{c}$ and $\bar{y}$ be the optimal solution to its dual problem \reff{mod:DLO-LCP}. 
One can easily verify that the optimal solution to model \reff{mod:scaled-DDLO1} can be obtained at $\bar{u}=\bar{x}, \bar{v}=\bar{y}$. Therefore, it holds
$\bar{c}^T \bar{u}=r^T\bar{v}$.  From the above discussion it follows
\begin{theorem}\label{thm2:PD-ReCR} Suppose that  the LO subproblem \reff{mod:DDLO} at every iteration of PD-ReCR is solved exactly with $\epsilon=0$.   If   the sequence $\{c^k\}$ generated by PD-ReCR is  unbounded,       then every accumulation point of the scaled sequence  $(\bar{c}^k,u^k, \bar{v}^k)$   is the global optimal solution of problem \reff{mod:scaled-DDLO1} where the optimal objective function value equals 0.
\end{theorem}

Theorem~\ref{thm2:PD-ReCR}  establishes the convergence of  the generated sequence provided from PD-ReCR when the sequence $\{ c^k\}$ is unbounded. One interesting question is whether we can identify such a scenario  in advance  so that we can terminate the algorithm early.
In what follows, we present a sufficient condition under which PD-ReCR will generate an unbounded sequence $\{c^k\}$ while the sequence $\{ u^k\}$ (or $\{x^k\}$) remains invariant.
From Theorem~\ref{thm-2:PD-ReCR} it follows
\begin{corollary}\label{cor3:PD-ReCR} Let $(u^k,v^k,s^k)$ be the optimal solution to model \reff{mod:DDLO} at iteration $k$.
Then $u^k=u^{k+1}$ if and only if
 $u^k$ is  the optimal solution to the following LO
\begin{eqnarray} \min &&\quad u^T s^k \label{Rec-Subprob}\\
 \mathrm{s.t.} &&\quad M u \ge r, \quad u\ge 0,  \nonumber \end{eqnarray}
 and $s^k=s^{k+1}$.
 Furthermore, if $u^k=u^{k+1}$, then
 \[  u^{k_1}=u^k, \quad s^{k_1}=s^k, \quad \forall k_1>k.\]
\end{corollary}

Corollary~\ref{cor3:PD-ReCR}  presents a sufficient condition under which the sequence $\{c^k\}$ generated by PD-ReCR is unbounded while the sequence $\{ (u^k, s^k\}$ remains invariant for sufficiently large $k$. Based on such a sufficient condition, we should terminate the algorithm whenever the equality $u^k=u^{k+1}$ holds, as continuing the PD-ReCR can not help to find a feasible solution to the underlying LCP. Under such a circumstance, we can solve  model \reff{mod:PLO-LCP} with $c=c^k$ to obtain a partial optimal solution.

We point out that though we have established the convergence of the basic PD-ReCR algorithm, it is possible that PD-ReCR  may provide only a partial optimal solution to the bilinear model \reff{mod:DDLO} that is not necessary a feasible solution to the original LCP. On the other hand, PD-ReCR works in the primal--dual setting and thus is not effective from a computational perspective. In the next subsection, we will introduce a new variant of ReCR to address these concerns in PD-ReCR.

\subsection{Rectified relaxation in the primal space (P-ReCR)}\label{subsec:P-ReCR}

In this subsection, we introduce a new variant  of ReCR that works in the primal or ($x$-) space.
To start, we mention that one potential setback in PD-ReCR is that we  need to solve problem  \reff{mod:DDLO} to rectify the vector $c$.
Note that  solving problem \reff{mod:DDLO} is more time-consuming than  problem \reff{mod:PLO-LCP}, due to  the increase in the number of  decision variables. It is desirable to rectify the vector $c$ using only the information from  the optimal solution of  problem \reff{mod:PLO-LCP}.

 Let  $(u^c, v^c)$ be the optimal solution to problem \reff{mod:DDLO}. Then one can see that $u^c$ is also a feasible solution to problem
 \reff{mod:PLO-LCP}. Moreover, the rectification direction $s^c$ in PD-ReCR can also by defined by \reff{def2:sc}.
 Let $x^p$ be the optimal solution to problem \reff{mod:PLO-LCP}.
 Let \[ \mathcal{I}^{x^p}_+=\{ i\in \mathcal{I}: x^p_i>0, (Mx^p)_i>r_i\}, \quad \mathcal{I}^{x^p}_0=\{ i\in \mathcal{I}: x^p_i=0.\},\]
  Based on the above observations, we propose to use the following rectification direction
 \begin{eqnarray}\label{def1:sp} s^p_i=\left\{ \begin{array}{ll}
	[Mx^p-r]_i &\quad \mbox{if $i\in \mathcal{I}^{x^p}_+$} \\
	0 &\quad \mbox{otherwise}. \end{array}\right.
 \end{eqnarray}
 when only information from the optimal solution   of  problem \reff{mod:PLO-LCP} is available.
 From the above definition it follows immediately
\begin{theorem}\label{thm:P-Rec-Func} The rectification vector $s^p\ge 0$  and $s^p=0$ if and only if $x^p$ is a feasible solution to the original LCP \reff{mod:LCP}.
\end{theorem}
We remark that the rectification vector $s^p$ enjoys similar properties as the so-called merit functions for complementary problems
\cite{Facchinei2002VICP, Fischer1995, Fukushima1996}. However, the derivation of these functions are very different. Most existing merit functions are constructed based on the component-wise complementary conditions, while the rectification functions in this paper are derived based on the optimal solution to   model \reff{mod:PLO-LCP}. It should be pointed out that the mapping $s^p$ is not continuous.

  We next introduce a safe-guide direction defined by
 \begin{eqnarray}\label{def:Safe-Guide} \hat{c}_i=\left\{ \begin{array}{ll}
	[M^T x^p]_i + \tau &\quad \mbox{if $i \in \mathcal{I}^{x^p}_0$} \\
	{[M^Tx^p]}_i &\quad \mbox{if $x^p_i>0$}. \end{array}\right.
                \end{eqnarray}
 Here $\tau$ is some penalty parameter to help keep the zero-valued  elements at the current point invariant at the next iteration. In numerical experiments, we may choose $\tau=10$.
 The above safe-guide direction is introduced to avoid the unnecessary large parameters associated with the zero-valued elements of $x^p$ that may be derived in some early stage of the ReCR process while allowing the usage of the rectification step $\max(M^T x^p,c)$ within a certain range.  It is a minor modification of the search direction $M^T x^p$ as suggested in the URT theory.
 Based on the above safe-guide direction, we propose to   rectify the vector $c$ via  the following scheme:
\begin{eqnarray}\label{Scheme:P-ReCR}
c^+=\min(\hat{c}, \max(M^Tx^p,c))+\theta s^p, \quad \theta\in (0,1],
\end{eqnarray}
where $\theta$ is a parameter to maintain the balance between the safe-guide direction and the rectification direction
Then we resolve \reff{mod:PLO-LCP} with $c=c^+$ and denote its optimal solution by $x^+$.   Such a process is repeated until convergence is reached.  Since the new algorithm works in the primal $x$-space, we call it primal ReCR or P-ReCR.

For simplicity of discussion, we assume   $\theta=1$ in the theoretical analysis regarding P-ReCR. From the choice of $s^p$, it follows immediately
 \begin{theorem}\label{thm-2:P-ReCR}
Let $x^p$ and $x^+$ be the exact optimal solutions to models \reff{mod:PLO-LCP}  with vector $c$ and $c^+$   respectively. Then $x^p=x^+$ if and only if $s^p=s^+$. Moreover, if $x^p=x^+$,  then $x^p$ is also a  stationary point of model \reff{mod:QO-LCP}.
 \end{theorem}
 \begin{proof} The first conclusion of the theorem follows directly from the choice of $s^p$ and $s^+$.  To prove the second conclusion of the theorem, we observe that
 \[  c^+= c^{+1}+c^{+2},\]
 where \[  {c}^{+1}_i=\left\{ \begin{array}{cc}
                \hat{c}_i & \mbox{ if $ i\not\in \mathcal{I}^{x^p}_+$}; \\
                0 & \mbox{if $ i\in \mathcal{I}^{x^p}_+$}. \end{array}\right. \qquad  c^{+2}_i=\left\{ \begin{array}{cc}
                0 & \mbox{ if $i\not\in \mathcal{I}^{x^p}_+ $}; \\
                (M^T x^p)_i+s^p_i & \mbox{if $i\in \mathcal{I}^{x^p}_+ $}. \end{array}\right. \]
Similarly, we can also rewrite the objective function in \reff{mod:QO-LCP} as
\[  x^T(Mx-r)=\sum_{i\in \mathcal{I}^{x^p}_+} x_i(Mx-r)_i +   \sum_{i\not \in \mathcal{I}^{x^p}_+} x_i(Mx-r)_i.\]
From the choice of $x^p$ we obtain
\[ \sum_{i\not\in \mathcal{I}^{x^p}_+} x_i^p(Mx^p-r)_i=0,\]
i.e., $x^p$ is the global optimal solution of minimizing $\sum_{i\not \in \mathcal{I}^{x^p}_+} x_i(Mx-r)_i$.
Now recall that
\[ \frac{\partial{ [x_i(Mx-r)_i]}}{\partial x_i}=c_i^{+2}, \quad \forall i \in \mathcal{I}^{x^p}_+.\]
 Combining the above relationship with the fact that $x^+$ is the optimal solution to problem \reff{mod:PLO-LCP} with $c=c^+$, we conclude that $x^p=x^+$ is a stationary point of the QO model \reff{mod:QO-LCP}. \end{proof}

 Now we are ready to describe a slightly modified version of the basic P-ReCR algorithm.

\par\noindent\rule{\textwidth}{0.4pt}
\noindent{\bf \hspace*{1.5cm}  Primal ReCR for LCP}
\par\noindent\rule{\textwidth}{0.4pt}
\begin{itemize} \itemsep-2pt
	\item[]{\bf Input:} \( M, r, \theta, \tau \), stop criteria \( \epsilon > 0 \), maximum iterations \( K \);
	\item[]{\bf Step 0.} Initialization
	\begin{itemize}
		\item[] {Step 0.1.} Set \( k = 0 \), \( c^k = \textbf{1} \);
		\item[] {Step 0.2.} Solve problem (\ref{mod:PLO-LCP}) to obtain \( x^k \);
		\item[]{Step 0.3.} If problem (\ref{mod:PLO-LCP}) is infeasible
		\item[] \qquad  \qquad  \quad  The LCP is infeasible and stop;
		\item[] \qquad \qquad Else If \( f_\mathrm{q}(x^k)\le \epsilon \) or \( \|\phi(x^k)\| \le \epsilon\)  then
		\item[] \qquad  \qquad  \quad  Go to Step 2;
		\item[] \qquad \qquad End if
		\item[] {Step 0.4.} Update \( k = k + 1 \);
		\item[] {Step 0.5.} Set \( c^k=M^T x^{k-1} \);
		\item[] {Step 0.6.} Solve problem (\ref{mod:PLO-LCP}) with \( c=c^k \) for \( x^k \);
		\item[]{Step 0.7.} If \( f_\mathrm{q}(x^k)\le \epsilon \)  or \( \|\phi(x^k)\| \le \epsilon\) then
		\item[] \qquad  \qquad  \quad  Go to Step 2;
		\item[] \qquad \qquad End if
	\end{itemize}
	\item[]{\bf Step 1.} P-ReCR process
	\begin{itemize}\itemsep-2pt
		\item[]{ Step 1.1.} Update \( k = k + 1 \);
		\item[] {Step 1.2.}  Update \( c^k \) by (\ref{Scheme:P-ReCR});
		\item[]{ Step 1.3.} Solve problem (\ref{mod:PLO-LCP}) to obtain \( x^k \);
		\item[]{ Step 1.4.} If \( f_\mathrm{q}(x^k)\le \epsilon \) or  \( \|\phi(x^k)\| \le \epsilon\) then
		\item[] \qquad  \qquad  \quad  Go to Step 2;
		\item[] \qquad \qquad End if
		\item[]{ Step 1.5.} If \( \| x^k-x^{k-1}\| \le \epsilon \) then
		\item[]{ Step 1.6.}  \quad Go to Step 2;
		\item[]{ Step 1.7.} Else
		\item[]{ Step 1.8.}  \quad Go to Step 1.1;
		\item[]{ Step 1.9.} End If
	\end{itemize}
\item[]{\bf Step 2.} Final Output and Feasibility Check.
\end{itemize}
\par\noindent\rule{\textwidth}{0.4pt}
We remark that scheme \reff{Scheme:P-ReCR} is similar to the predictor-corrector method \cite{Mehrotra1992IPM}, a special variant of interior-point methods for linear and conic optimization. Like in the numerical implementation for predictor-corrector method,  we also   choose $\theta$ to be a small number $\theta= 0.1$ in the implementation of ReCR.

We next  investigate the convergence of the  P-ReCR with $\theta=1$. For this, we make the following assumption:
\begin{assumption}\label{Feas-PLO} The LO model \reff{mod:PLO-LCP} is strictly feasible.
\end{assumption}
We remark that the above assumption can be made w.l.o.g. To see this, we observe that
the issue of determining whether model \reff{mod:PLO-LCP} is strictly feasible can be addressed via solving the following LO
\begin{eqnarray*}
	\max &&\quad t \\
	\text{s.t.} &&\quad Mx-t\mathbf{1}\ge r, \quad x\ge 0, \quad t\ge 0.
\end{eqnarray*}
Moreover, when model \reff{mod:PLO-LCP} is  feasible but not strictly feasible, we can further solve several LOs to identify the index set $\mathcal{I}^1$ such that the following equalities are satisfied
\[ (Mx-r)_i=0, \quad \forall i\in \mathcal{I}^1.\]
Under such a circumstance, the underlying LCP \reff{mod:LCP} can be further simplified.

Now we study the convergence behavior of P-ReCR. For this, we consider an accumulation point of the sequence generated by ReCR. Since the optimal solution to problem \reff{mod:PLO-LCP} is invariant for scaling  the coefficient vector, thus we can further assume $c^k$ is bounded. Let $\bar{c}$ denote an accumulation point of the scaled sequence $c^k$ and $\bar{x}$ be the optimal solution to model \reff{mod:PLO-LCP} with $c=\bar{c}$. W.l.o.g., we assume
\[  \bar{c}=\lim_{k\infty} c^k.\]
Let $\bar{x}$ be the optimal solution to model \reff{mod:PLO-LCP} with $c=\bar{c}$. We say $\bar{x}$ an non-degenrate accumulation point if $\bar{x}$ is a non-degenerate solution of model \reff{mod:PLO-LCP}. We have
\begin{theorem}\label{mainthm1:P-ReCR} Suppose that Assumption~\ref{Feas-PLO} holds. Then  every non-degenerate accumulation point of the sequence generated from P-ReCR is    either a feasible solution to the original LCP or a stationary point of the QO model \reff{mod:QO-LCP} within $\epsilon$-tolerance. 
\end{theorem}
\begin{proof}
Since $\bar{x}$ is a non-degenerate solution of model \reff{mod:PLO-LCP} with $c=\bar{c}$, from the sensitivity analysis in \cite{Holder1997LO-SA} it follows immediately
\[ \bar{x} =\lim_{k\infty} x^k.\]
Combining the above relationship with Theorem~\ref{thm-2:P-ReCR}, we obtain the conclusion in the theorem and thus
 the proof is completed.  \end{proof}

We remark that the assumption that $\bar{x}$ being a non-degenerate solution of model \reff{mod:PLO-LCP} is rather mild. This is because in P-ReCR, we rely mainly on the safe-guide direction, which increase the coefficient values associated with all the nonbasic variables $\{ x_i:i\in \mathcal{I}_0\}$ so that all the nonbasic variables have strictly positive reduced costs.

We next discuss how to obtain a certificate for the infeasibility of the original LCP when the solution provided by P-ReCR is not a feasible solution to the original LCP. For this, let define
\[
\bar{\mathcal{I}}_+=\{i\in \mathcal{I}: \bar{x}_i>0,\ (M\bar{x}-r)_i>0\}; \quad \bar{\mathcal{I}}_0=\mathcal{I} \setminus \bar{\mathcal{I}}_+.
\]
Let \(\bar{\mathcal{I}}_+=\{i_1,\dots,i_s\}\), where \(s=|\bar{\mathcal{I}}_+|\). Observe that the complementarity condition
\[
x_i(Mx-r)_i=0
\]
is equivalent to the logical disjunction
\[
x_i=0
\qquad \text{or} \qquad
(Mx-r)_i=0.
\]
Therefore, every feasible solution to the original LCP must satisfy one complementarity branch for every index \(i\in\bar{\mathcal{I}}_+\).

For every binary vector
\[
\delta=(\delta_1,\dots,\delta_s)\in\{0,1\}^s,
\]
we define the following LP feasibility problem associated with one complete complementarity pattern:
\begin{subequations} \label{mod:Infeas-LCP}
 \begin{eqnarray}
 \quad \min &&\quad 0 \\
 \text{s.t.} &&\quad Mx \ge r,\quad x\ge 0, \nonumber\\
 &&\quad x_{i_\ell}=0 \quad \text{if } \delta_\ell=0,\qquad \ell=1,\dots,s,\nonumber\\
 &&\quad (Mx-r)_{i_\ell}=0 \quad \text{if } \delta_\ell=1,\qquad \ell=1,\dots,s. \nonumber
 \end{eqnarray}
\end{subequations}

Since every feasible LCP solution must satisfy exactly one admissible complementarity branch at each violated index, the original LCP is feasible if and only if at least one problem \reff{mod:Infeas-LCP} is feasible. This immediately yields the following result.

\begin{proposition}\label{prop:Infeas}
If the original LCP is feasible, then   there exists a binary vector
$\delta\in\{0,1\}^{|\bar{\mathcal{I}}_+|}$
such that the associate branch problem \reff{mod:Infeas-LCP} is feasible.  If all the problems in \reff{mod:Infeas-LCP}  are infeasible, then the original LCP is infeasible.
\end{proposition}

\begin{proof}
	Suppose first that the original LCP is feasible. Then there exists a point \(x^*\) such that
	\[
	Mx^*\ge r,\qquad x^*\ge 0,
	\]
	and
	\[
	x_i^*(Mx^*-r)_i=0,\qquad i\in\mathcal I.
	\]
	Since \(\bar{\mathcal I}_+=\{i\in\mathcal I:\bar x_i>0,\ (M\bar x-r)_i>0\}\), for every \(i\in\bar{\mathcal I}_+\) the complementarity condition implies that at least one of the following two equalities holds:
	\[
	x_i^*=0
	\qquad \text{or} \qquad
	(Mx^*-r)_i=0.
	\]
	Define a binary vector \(\delta=(\delta_i)_{i\in\bar{\mathcal I}_+}\in\{0,1\}^{|\bar{\mathcal I}_+|}\) by
	\[
	\delta_i=
	\begin{cases}
		0, & \text{if } x_i^*=0,\\
		1, & \text{if } (Mx^*-r)_i=0.
	\end{cases}
	\]
	Then \(x^*\) satisfies the constraints of the corresponding branch problem \(\reff{mod:Infeas-LCP}\). Hence, the branch problem associated with \(\delta\) is feasible.
	The second conclusion follows directly from the first conclusion in the proposition. \end{proof}

We remark that Proposition~\ref{prop:Infeas} provides a theoretical guarantee for infeasibility detection. In the worst case, this procedure may require solving up to \(2^{|\bar{\mathcal I}_+|}\) LP feasibility problems. For computational consideration, we apply the recursive disjunctive search only after a relatively small active set \(\bar{\mathcal I}_+\) has been identified by P-ReCR.

Now we are ready to describe the detailed Step 2 in ReCR. 
\par\noindent\rule{\textwidth}{0.4pt}
	\noindent{\bf \hspace*{1.5cm} Step 2. Final Output and Feasibility Check}
\par\noindent\rule{\textwidth}{0.4pt}
\begin{itemize}\itemsep-2pt
        \item[]{ Step 2.1}  If \(f_\mathrm{q}(x^*) \le  \epsilon\) or \( \|\phi(x^*)\| \le \epsilon\) then
          \item[]{Step 2.2}     \quad             Output \(x^* := x^k \) as the final solution;
        \item[]{Step 2.3} Else
         \item[]{}\qquad  \qquad \quad For each \(\delta\in\{0,1\}^{|\bar{\mathcal I}_+|}\)
          \item[]{} \qquad \qquad  \qquad Solve problem \reff{mod:Infeas-LCP} associated with the branch pattern \(\delta\);
           \item[]{} \qquad \qquad \quad If problem \reff{mod:Infeas-LCP}  is feasible then
           \item[]{} \qquad \qquad \qquad Output the current  solution as the final solution;
           \item[]{} \qquad \qquad \quad End if
           \item[]{} \qquad \qquad \quad End For
           \item[]{} \qquad \qquad \quad If all the branch problems \reff{mod:Infeas-LCP} are infeasible then
           \item[]{} \qquad \qquad \qquad The original LCP is infeasible;
           \item[]{} \qquad \qquad \quad End if
        \item[]{Step 2.4}   End If
    \end{itemize}
\par\noindent\rule{\textwidth}{0.4pt}

We finish this subsection by showing how P-ReCR works for an instance of infeasible LCP.
Consider the following LCP instance:
\begin{example}\label{Inf-LCP1}
\[
M=
\begin{pmatrix}
1 & 1 & 1 \\
1 & -2 & 1 \\
1 & 1 & -1
\end{pmatrix}
,\quad
r=
\begin{pmatrix}
-4 \\[4pt]
1 \\[4pt]
0
\end{pmatrix}.
\]
\end{example}

P-ReCR finds a solution \(x^* = (0.5,\,0,\,0.5)^\top\) to model \reff{mod:PLO-LCP} in 4 iterations. Since \(f_q(x^*)=2.5>0\), we need to check whether the underlying LCP is infeasible. Note that we have
\[
\bar{\mathcal I}_+=\{1\}.
\]
Hence, there exist two complementarity branch systems corresponding to
\[
x_1=0
\]
and
\[
(Mx-r)_1=0.
\]
By solving the two branch problems \reff{mod:Infeas-LCP}, we verify that both systems are infeasible. Therefore, Proposition~\ref{prop:Infeas} implies that the original LCP is infeasible. This demonstrates that P-ReCR has successfully detected the infeasibility of the underlying LCP.

\subsection{Rectified convex quadratic relaxation (ReQR)}
 We next introduce a new dynamic variant of ReCR which treats $c$ as a vector of parameters that may change in model \reff{mod:PDLO}.
For this, we recall that from the proof of Theorem ~\ref{thm:URT}, we can conclude that if the original LCP is feasible and let $x^*$ be a feasible solution to the underlying
LCP, then $c^*=M^Tx^*$.
Motivated by this observation, we propose to replace the constraint $M^Ty\le c$ in \reff{mod:PDLO} by $M^Ty\le M^Tx$, leading to  the following LO:
\begin{subequations} \label{mod:PDLO-1}
 \begin{eqnarray}  \quad \min &&\quad  c^Tx  \\
    \text{s.t.} &&\quad    Mx \ge r,  \quad  M^Ty \le M^Tx \\
        &&\quad           x\le  y, \quad x,y\ge 0.
  \end{eqnarray}
\end{subequations}

We now discuss how to solve model \reff{mod:PDLO-1}. First, we observe that  we can rewrite the objective function in  model \reff{mod:PDLO-1} as $(c-r)^Tx+r^T(x-y)$. Therefore, we can solve model \reff{mod:PDLO-1} via decomposing it into the following two smaller models:
\begin{subequations} \label{mod:PLO-1}
 \begin{eqnarray}  \min &&\quad  (c-r)^Tx \\
    \text{s.t.} &&\quad Mx \ge r, \quad  x\ge 0.
\end{eqnarray}
\end{subequations}
and
\begin{subequations} \label{mod:DLO-1}
 \begin{eqnarray}  \min &&\quad   r^T(x-y) \\
    \text{s.t.} &&\quad    M^T(x-y) \ge 0,  \quad  y-x\ge 0.
  \end{eqnarray}
\end{subequations}
Note that model \reff{mod:DLO-1} needs to  be solved only once at the very beginning. Particularly, we have the following proposition.
\begin{proposition} If problem \reff{mod:DLO-1} is bounded, then it must have a trivial optimal solution satisfying $x=y$. If problem \reff{mod:DLO-1} is unbounded,
then the original LCP \reff{mod:LCP} is infeasible.
\end{proposition}
 Let $x^c$ be the optimal solution to problem \reff{mod:PLO-1}.  we propose to rectify the vector $c$ as below:
\begin{eqnarray} \label{P-ReCR1} c^+=M^T x^c.\end{eqnarray}
One can easily verify that if the above scheme is adopted, then the dynamic ReCR can be viewed as a specific method for solving problem \reff{mod:QO-LCP}.
We can rewrite the objective function in \reff{mod:QO-LCP} as the following:
 \[ x^TM x -r^T x=\frac{x^T(M+M^T)x}{2}-r^Tx.\]
 Since for generic LCPs, the matrix $M+M^T$ may be indefinite. Under such a circumstance,  problem \reff{mod:QO-LCP} reduces to non-convex quadratic optimization with linear constraints.
 One popular way to solve non-convex QO is the so-called DC program \cite{Horst1999DC}. Particularly, Thi and Dinh \cite{Thi2011DC-LCP} developed various DC approaches for LCPs.
  Following the ideas in \cite{Thi2011DC-LCP}, we can rewrite the objective function in \eqref{mod:QO-LCP}  as  the following:
  \[ \frac{x^T(M+M^T)x}{2}-r^Tx= \frac{x^TM^+x}{2}-\frac{x^TM^-x}{2}-r^Tx,\]
  where $M^+, M^-$ are two symmetric positive semi-definite matrices derived from the singular value decomposition  of the matrix $M+M^T$ satisfying the following property
  \[ M^+\succeq0, M^-\succeq 0, \  M^+ M^-=0.\]
 Suppose that $x^k$ is available at iterate $k$, we then solve the following convex QO
 \begin{subequations} \label{mod:DC-1}
 \begin{eqnarray}  \min &&\quad  \frac{x^TM^+ x}{2}  - (x^k)^TM^-x- r^Tx \\
    \text{s.t.} &&\quad    Mx \ge r,  \quad          x\ge 0.
  \end{eqnarray}
\end{subequations}
Correspondingly, we call it   Rectified  Convex Quadratic Relaxation (ReQR).

Now we are ready to describe the ReQR for  LCP.

\par\noindent\rule{\textwidth}{0.4pt}
\noindent{\bf\hspace*{1.5cm}   Rectified  Quadratic  Relaxation for LCP}
\par\noindent\rule{\textwidth}{0.4pt}
\begin{itemize} \itemsep-2pt
	\item[]{\bf Input:} \( M, r \), stop criteria \( \epsilon > 0 \), maximum iterations \( K \);
	\item[]{\bf Step 0.} Initialization
	\begin{itemize}\itemsep-1pt
		\item[]{Step 0.1.} Compute \(M^+\) and \(M^-\) using singular value decomposition;
		\item[]{Step 0.2.} Set \( k = 0 \), \( c^k = \textbf{1} \);
		\item[]{Step 0.3.} Solve problem (\ref{mod:PLO-LCP}) to obtain \( x^k \);
		\item[]{Step 0.4.} If problem (\ref{mod:PLO-LCP}) is infeasible
		\item[] \qquad \qquad \quad The LCP is infeasible and stop;
		\item[] \qquad \qquad Else If \( f_\mathrm{q}(x^k)\le \epsilon \) or \( \|\phi(x^k)\| \le \epsilon\) then
		\item[] \qquad \qquad \quad Go to Step 2;
		\item[] \qquad \qquad End if
		\item[]{Step 0.5.} Update \( k = k + 1 \);
		\item[]{Step 0.6.} Solve problem (\ref{mod:DC-1}) to obtain its optimal solution \( x^k \);
		\item[]{Step 0.7.} If \( f_\mathrm{q}(x^k)\le \epsilon \) or \( \|\phi(x^k)\| \le \epsilon\) then
		\item[] \qquad \qquad \quad Go to Step 2;
		\item[] \qquad \qquad End if
	\end{itemize}
	\item[]{\bf Step 1.} Solving process
	\begin{itemize}\itemsep-1pt
		\item[]{Step 1.1.} Update \( k = k + 1 \);
		\item[]{Step 1.2.} Solve problem (\ref{mod:DC-1}) to obtain \( x^k \);
		\item[]{Step 1.3.} If \( f_\mathrm{q}(x^k)\le \epsilon \) or \( \|x^{k}-x^{k-1}\|\le \epsilon \) then
		\item[] \qquad \qquad \quad Go to Step 2;
		\item[] \qquad \qquad Else
		\item[] \qquad \qquad \quad Go to Step 1.1;
		\item[]{Step 1.4.} End if
	\end{itemize}
\item[]{\bf Step 2.} Final Output and Feasibility Check.
\end{itemize}
\par\noindent\rule{\textwidth}{0.2pt}

The following result follows from the convergence of DC program in \cite{Horst1999DC, Thi2011DC-LCP} and thus  its proof is omitted.
\begin{theorem}\label{thm:Q-ReCR}
 Suppose that  the QO subproblem at every iteration of ReQR is solved exactly and $\epsilon=0$ and Assumption~1 holds.  Then   the sequence $\{x^k\}$ generated by the  ReQR is  bounded and converges to a stationary point of  problem \reff{mod:QO-LCP}.
\end{theorem}

We finish this subsection by introducing the hybrid ReCR (H-ReCR). First we mention that though   P-ReCR is globally convergent, we have observed in numeric experiments that while P-ReCR can find a feasible solution to all the LCP instances  collected from the literature and most synthetic instances generated in this work, it struggled to either find a feasible solution or obtain a certificate for the infeasibility for some SLCC instances with non-convex solution set. One possible reason for this phenomena
  is that P-ReCR is essentially a first-order method. As  pointed out in several studies \cite{Beck2017FOM,Boyd2011ADM},  even for convex optimization problems, first-order methods may have a hard time to obtain a solution with high accuracy. However, to determine whether a given SLCC is feasible, we need to obtain a highly accurate solution to the constructed bilinear model.
    To  help obtain a highly accurate solution to the underlying bilinear model,    we propose to switch to the high-order ReQR method  if one of the following two thresholds \begin{eqnarray}\label{Switch-Rates} \rho^k_{1}=\frac{f_q(x^k)\!+\!\|\mathcal{I}_+^k\|}{f_q(x^0)\!+\!\|\mathcal{I}_+^0\|}\le 0.1\%, \   \rho^k_{2}=\frac{f_q(x^{k-1})\!+\!\|\mathcal{I}_+^{k-1}\|-f_q(x^k)\!-\!\|\mathcal{I}_+^k\|}{f_q(x^{k-1})\!+\!\|\mathcal{I}_+^{k-1}\|} \le 0.1\% \nonumber\end{eqnarray}
is met, where the first threshold is defined by the absolute improvement rate and the second threshold is defined by the relative improvement rate at iteration $k$.
In other words, whenever a good approximate but not highly accurate solution is generated from P-ReCR, we switch to ReQR.

Now we are ready to describe the H-ReCR for LCPs as follows.

\par\noindent\rule{\textwidth}{0.4pt}
\noindent{\bf\hspace*{1.5cm}   Hybrid ReCR for LCP }
\par\noindent\rule{\textwidth}{0.2pt}
\begin{itemize} \itemsep-2pt
	\item[]{\bf Input:} \( M, r, \theta, \tau \), stop criteria \( \epsilon > 0 \), maximum iterations \( K \);
	\item[]{\bf Step 0.} Initialization
	\begin{itemize}\itemsep-2pt
		\item[]{Step 0.1.} Compute \(M^+\) and \(M^-\) using singular value decomposition;
		\item[]{Step 0.2.} Set \( k = 0 \), \( c^k = \textbf{1} \);
		\item[]{Step 0.3.} Solve problem (\ref{mod:PLO-LCP}) to obtain \( x^k \);
		\item[]{Step 0.4.} If problem (\ref{mod:PLO-LCP}) is infeasible
		\item[] \qquad \qquad \quad The LCP is infeasible and stop;
			\item[] \qquad\qquad    Else If \( f_q(x^k)\le \epsilon \) or \( \|\phi(x^k)\| \le \epsilon\) then
		\item[] \qquad \qquad \quad Go to Step 3;
		\item[] \qquad \qquad End if
		\item[]{Step 0.6.} Update \( k = k + 1 \), \( c^k=M^T x^{k-1} \);
		\item[]{Step 0.7.} Solve problem (\ref{mod:PLO-LCP}) with \( c=c^k \) for \( x^k \);
		\item[]{Step 0.8.} If \( f_\mathrm{q}(x^k)\le \epsilon \) or \( \|\phi(x^k)\| \le \epsilon\)  then
		\item[] \qquad \qquad \quad Go to Step 3;
			\item[]\qquad \qquad Else If \( \rho^k_{1}< 0.01 \) or $\rho^k_{2}\le 0.01$ then
		\item[] \qquad \qquad \quad Go to Step 2;
		\item[] \qquad \qquad End if
	\end{itemize}
	\item[]{\bf Step 1.} P-ReCR process
	\begin{itemize}\itemsep-2pt
		\item[]{Step 1.1.} Update \( k = k + 1 \), update \( c^k \) by (\ref{Scheme:P-ReCR});
		\item[]{Step 1.2.} Solve problem (\ref{mod:PLO-LCP}) to obtain \( x^k \);
		\item[]{Step 1.3.} If \( f_\mathrm{q}(x^k)\le \epsilon \) or \( \|\phi(x^k)\| \le \epsilon\) then
		\item[] \qquad \qquad \quad Go to Step 3;
		\item[] \qquad \qquad End if
		\item[]{Step 1.4.} If \( \rho^k_{abs}< 0.01 \) or $\rho^k_{rel}\le 0.01$ then
		\item[] \qquad \qquad \quad Go to Step 2;
		\item[]{Step 1.5.} Else if \( \|x^{k}-x^{k-1}\| \le \epsilon \)  then
		\item[] \qquad \qquad \quad Go to Step 3;
		\item[]{Step 1.6.} Else
		\item[] \qquad \qquad \quad Go to Step 1.1;
		\item[]{Step 1.7.} End if
	\end{itemize}
	\item[]{\bf Step 2.} ReQR process
	\begin{itemize}\itemsep-2pt
		\item[]{Step 2.1.} Update \( k = k + 1 \);
		\item[]{Step 2.2.} Solve problem (\ref{mod:DC-1}) using $x^{k-1}$ to obtain \( x^k \);
		\item[]{Step 2.3.} If \( f_\mathrm{q}(x^k)\le \epsilon \) or \( \|\phi(x^k)\| \le \epsilon\) then
		\item[] \qquad \qquad \quad Go to Step 3;
		\item[] \qquad \qquad End if
		\item[]{Step 2.4.} If \( \|x^{k}-x^{k-1}\| \le \epsilon \) then
		\item[] \qquad \qquad \quad Go to Step 3;
		\item[]{Step 2.5.} Else
		\item[] \qquad \qquad \quad Go to Step 2.1;
		\item[]{Step 2.6.} End if
	\end{itemize}
\item[]{\bf Step 3.} Final Output and Feasibility Check.
\end{itemize}
\par\noindent\rule{\textwidth}{0.2pt}

\section{Rectified Convex Relaxations for Generic SLCCs}
\label{sec:ReCR-SLCC}

In this section, we extend the URT and two ReCR variants from LCPs to the SLCC model~\reff{mod:SLCC}. We denote the solution set of \reff{mod:SLCC} as $\mathcal{X}_{\mathrm{SLCC}}$.
Our task here is to extend the ReCR approaches to find a solution $(x^*,z^*)\in \mathcal{X}_{\mathrm{SLCC}}$ or determine the set $\mathcal{X}_\mathrm{SLCC}$ is empty.

To study feasibility of model \reff{mod:SLCC}, we temporarily fix vectors \((c_x,c_z)\) and consider the following LO model:
\begin{subequations}\label{mod:PLO-SLCC}
\begin{eqnarray}
\min &&\quad c_x^T x + c_z^T z \\[.2em]
\text{s.t.} &&\quad Mx + Qz \geq r \\
&&\quad Ax + Bz = b \\
&&\quad x,z \geq 0.
\end{eqnarray}
\end{subequations}
 The dual of the above problem reads as
\begin{subequations}\label{mod:DLO-SLCC}
\begin{eqnarray}
\max &&\quad r^T u + b^T v \\[.2em]
\text{s.t.} &&\quad M^T u + A^T v \leq c_x \\
&&\quad Q^T u + B^T v = c_z \\
&&\quad u\geq 0, \; v\text{ free}.
\end{eqnarray}
\end{subequations}
Similar to Theorem~\ref{thm:URT}, we have the following URT for SLCCs.
\begin{theorem}\label{thm:URT-SLCC}
The set \(\mathcal{X}_{\mathrm{SLCC}}\) is nonempty if and only if there exist vectors \((c_x,c_z)\) such that the primal--dual pair \eqref{mod:PLO-SLCC}--\eqref{mod:DLO-SLCC} admits an optimal solution \((x^c,z^c;u^c,v^c)\) satisfying \(x^c \leq u^c\) component-wise.
\end{theorem}
\begin{proof}
    The sufficiency part of the theorem follows directly from the duality theory for LO. It remains to prove the necessity part of the theorem.

    Suppose $(x^*,z^*) \in \mathcal{X}_{\mathrm{SLCC}}$. Choose $c_x=M^Tx^*$, and $c_z=Q^T x^*$. Then $(u^*, v^*) = (x^*, 0)$ is feasible for the dual problem \reff{mod:DLO-SLCC}. Moreover, since $(x^*)^T(Mx^*+Qz^*-r)=0$, it follows immediately $c_x^Tx^*+c_z^Tz^*=r^Tx^*$. Hence, $(x^*,z^*)$ and $(u^*,v^*) = (x^*,0)$ are primal--dual optimal for \reff{mod:PLO-SLCC}--\reff{mod:DLO-SLCC}. This completes the proof of the theorem.
\end{proof}

Based on  Theorem \ref{thm:URT-SLCC}, we can address \reff{mod:SLCC} through the following bilinear optimization problem, in which $(c_x,c_z)$ are also decision variables:
\begin{subequations}\label{mod:PDLO-SLCC}
\begin{eqnarray}
\min &&\quad c_x^T x + c_z^T z - r^T u - b^T v \\[.2em]
\text{s.t.} &&\quad Mx + Qz \geq r,\; Ax + Bz = b,\label{eq:PDLO-P}\\
&&\quad M^T u + A^T v \leq c_x,\; Q^T u + B^T v = c_z,\label{eq:PDLO-D}\\
&&\quad x \leq u,\; x,z,u\geq 0,\; v\text{ free}.\label{eq:PDLO-link}
\end{eqnarray}
\end{subequations}
The following proposition is a direct consequence of Theorem \ref{thm:URT-SLCC}.
\begin{proposition}\label{prop:PDLO-SLCC}
Let \((c_x^{*}, c_z^{*}, x^{*}, z^{*}, u^{*}, v^{*})\) be a global minimizer of \eqref{mod:PDLO-SLCC}. If
\[
c_x^{*T} x^{*} + c_z^{*T} z^{*} - r^T u^{*} - b^T v^{*} = 0,
\]
then \((x^{*},z^{*}) \in \mathcal{X}_{\mathrm{SLCC}}\); otherwise, \(\mathcal{X}_{\mathrm{SLCC}} = \varnothing\).
\end{proposition}
Similar optimality conditions can be derived by fixing $(c_x,c_z)$ in \reff{mod:PDLO-SLCC} and analyzing the corresponding primal--dual LP pair.
Yet for computational consideration, we design SLCC ReCR methods using the optimal solution of the primal LO model~\reff{mod:PLO-SLCC}.

\subsection{Primal rectified convex relaxation for SLCCs}\label{subsec:P-ReCR-SLCC}

An important step in the development of ReCR approaches is how to choose the rectification vector.  For fixed vectors \((c_x,c_z)\),  let \((x^p,z^p)\) be a optimal solution of \eqref{mod:PLO-SLCC} satisfying strict complementarity.  As in P-ReCR for LCPs, define the rectification vector \( s^p \) component-wise by
\begin{equation}\label{Scheme:P-ReCR-SLCC}
s^p_i =
\begin{cases}
(Mx^p + Qz^p - r)_i \quad &\text{if } x^p_i > 0, \ (Mx^p + Qz^p - r)_i > 0 \\[6pt]
0 & \text{otherwise.}
\end{cases}
\end{equation}

\noindent We immediately have the following result.
\begin{theorem}
The rectification vector \( s^p \geq 0 \) and \( s^p = 0 \) if and only if \( (x^p,z^p) \) is a feasible solution to the original SLCC \reff{mod:SLCC}.
\end{theorem}
Like in the case for LCP,  we introduce the following safe-guide direction for SLCC:
\begin{eqnarray}\label{def:Safe-Guide-SLCC} \hat{c}_i=\left\{ \begin{array}{ll}
    (M^T x^p)_i + \tau & \mbox{ if $i \in \mathcal{I}^{x^p}_0$} \\
    (M^Tx^p)_i & \mbox{ if $x^p_i>0$.} \end{array}\right.
\end{eqnarray}
Then we    propose to   rectify the vector $c$ via  the following scheme:
\begin{eqnarray}\label{P-ReCR:SLCC}
c_x^+=\min(\hat{c}, \max(M^Tx^p,c_x))+s^p.
\end{eqnarray}
From the above rectification scheme and the choice of $s^p$, it follows immediately
\begin{theorem}\label{thm-2:P-ReCR-SLCC}
Let $(x^p,z^p)$ and $(x^+,z^+)$ be exact optimal solutions of \reff{mod:PLO-SLCC} with cost vectors $(c_x,c_z)$ and $(c_x^+,c_z^+)$, respectively.
Then $x^p=x^+$ if and only if $s^p=s^+$. Moreover, if $x^p=x^+$,  then $x^p$ is also a  stationary point of model \reff{mod:QO-SLCC}.
 \end{theorem}
 \begin{proof} The proof of the theorem follows a similar vein as that of Theorem~\ref{thm-2:P-ReCR} and thus the detail is omitted.\end{proof}

The new ReCR approach for SLCCs works as follows. First we use a similar strategy like in P-ReCR for LCP to find a starting solution. Then we resolve problem \eqref{mod:PLO-SLCC} using the rectified vector \( (c_x^{k}, c_z^k) \) and repeat the process  until a stopping criterion is met.
Since this algorithm operates only in the primal space, we name it \textit{primal ReCR for SLCC} (P-ReCR-SLCC).

\par\noindent\rule{\textwidth}{0.4pt}
	\noindent{\bf \hspace*{1.5cm} Primal Rectified Convex Relaxation for SLCC}
\par\noindent\rule{\textwidth}{0.4pt}
\begin{itemize}\itemsep-2pt
	\item[]{\bf Input:} \( M, Q, A, B, r, b, \theta, \tau \), tolerance \( \epsilon > 0 \), maximum iterations \( K \);
	\item[]{\bf Step 0.} Initialization
	\begin{itemize}\itemsep-2pt
        \item[]{ Step 0.1.} Set \( k = 0 \), \( c_x^{k} = \mathbf{1} \), \( c_z^k = \mathbf{0} \);
		\item[]{ Step 0.2.} Solve \eqref{mod:PLO-SLCC} with $(c_x,c_z) = (c^{k}, c_z^k)$ and obtain \( (x^{k}, z^{k}) \);
        \item[]{Step 0.3.} If problem \ref{mod:PLO-SLCC} is infeasible then
        \item[]{} \qquad \qquad \quad The SLCC is infeasible;
        \item[]{} \qquad \qquad Else if $ (x^{k})^T(Mx^k+Qz^k-r)\le \epsilon$ or $|\mathcal{I}_+(x^k,z^k)|=0$, then
        \item[]{} \qquad \qquad \quad Output $(x^k, z^k)$ as the final solution;
        \item[]{}\qquad \qquad End if
        \item[] { Step 0.4.} Update \( k = k+1 \);
        \item[] { Step 0.5.} Set \( c_x^{k} = M^T x^{k-1}, c_z^k=Q^T x^{k-1} \);
	    \item[] { Step 0.6.} Solve \eqref{mod:PLO-SLCC} with \( (c_x^k, c_z^k) \) to obtain \( (x^{k}, z^{k}) \);
	\end{itemize}
		\item[]{\bf Step 1.} Solving process
	\begin{itemize}\itemsep-2pt
        \item[]{ Step 1.1.} Update \( k = k + 1 \);
        \item[]{ Step 1.2.} Compute \( s^p \) based on \eqref{Scheme:P-ReCR-SLCC};
        \item[]{ Step 1.3.} Rectify the vector \( c \) by \ref{P-ReCR:SLCC};
		\item[]{ Step 1.4.} Solve \eqref{mod:PLO-SLCC} with  \( (c_x^{k}, c_z^k) \); obtain \( (x^{k}, z^{k}) \);
        \item[]{ Step 1.5.} If \( \| (x^{k}, z^{k}) - (x^{k-1}, z^{k-1}) \| \geq \epsilon \), go to Step 1.1;
	\end{itemize}
		\item[]{\bf Step 2.} Output \( (x^*, z^*) := (x^{k}, z^{k}) \) and check feasibility.
\end{itemize}
\par\noindent\rule{\textwidth}{0.4pt}

Next we study the convergence of P-ReCR-SLCC under the following assumptions.
\begin{assumption}\label{Assump3}
      The global optimal solution to
      \begin{subequations}\label{mod:QO-SLCC}
      \begin{eqnarray}
      \min &&\quad x^{\!T}(Mx+Qz-r)\\
      \mathrm{s.t.} &&\quad Mx+Qz\ge r, \; Ax+Bz=b,\; x,z\ge0
      \end{eqnarray}
      \end{subequations}
      is bounded.
      \end{assumption}
\begin{assumption}\label{Assump4}
      For every cost vector $c\ge0$, the primal LO
      \begin{subequations}\label{mod:PLO-SLCC-c}
      \begin{eqnarray}
      \min &&\quad c^Tx\\
      \mathrm{s.t.} &&\quad Mx+Qz\ge r,\; Ax+Bz=b,\; x,z\ge0
      \end{eqnarray}
      \end{subequations}
      is strictly feasible.
\end{assumption}
The following theorem can be proved via following a similar vein as in the proofs of their counterparts Theorem~\ref{mainthm1:P-ReCR} for LCP. Therefore, we omit the detailed proofs here.
\begin{theorem}\label{mainthm1:P-ReCR-SLCC} Suppose that Assumption~\ref{Assump4} holds. Then P-ReCR-SLCC terminates, within $\epsilon$-tolerance, at either a feasible solution of the original SLCC or a stationary point of the QO model.
\end{theorem}

Finally, we consider the scenario when P-ReCR-SLCC can only provide a solution \((x^*,z^*)\), which is not a solution of the original SLCC \reff{mod:SLCC}. Under such a circumstance, we need to check the infeasibility of the underlying SLCC. Let
\[
\mathcal{I}_+
      =\bigl\{i\in\mathcal{I}\; \bigl|\;
          x_i^{*}>0,\;
          (Mx^{*}+Qz^{*}-r)_i>0\bigr\},
\qquad
\bar{\mathcal{I}}_+=\mathcal{I}\setminus\mathcal{I}_+ .
\]
Let \(\mathcal{I}_+=\{i_1,\dots,i_s\}\), where \(s=|\mathcal I_+|\). Observe that the complementarity condition
\[
x_i(Mx+Qz-r)_i=0
\]
is equivalent to the logical disjunction
\[
x_i=0
\qquad \text{or} \qquad
(Mx+Qz-r)_i=0.
\]
Hence, every feasible solution of the original SLCC must satisfy one complementarity branch for every index \(i\in\mathcal I_+\).

For every binary vector
\[
\delta=(\delta_1,\dots,\delta_s)\in\{0,1\}^s,
\]
we define the following LP feasibility problem associated with one complete complementarity pattern:
\begin{subequations}\label{mod:Infeas-SLCC-1}
\begin{eqnarray}
\quad \min &&\quad 0 \\
\text{s.t.} &&\quad Mx+Qz\ge r,\quad x,z\ge 0,\nonumber\\
&&\quad Ax+Bz=b,\nonumber\\
&&\quad x_{i_\ell}=0 \quad \text{if } \delta_\ell=0,\qquad \ell=1,\dots,s,\nonumber\\
&&\quad (Mx+Qz-r)_{i_\ell}=0 \quad \text{if } \delta_\ell=1,\qquad \ell=1,\dots,s.\nonumber
\end{eqnarray}
\end{subequations}

Since every feasible SLCC solution must satisfy exactly one admissible complementarity branch at each violated index, the original SLCC is feasible if and only if at least one problem \reff{mod:Infeas-SLCC-1} is feasible.

This immediately yields the following result.

\begin{proposition}\label{prop:Infeas-SLCC}
If the original SLCC is feasible, then there exists a binary vector
\[
\delta\in\{0,1\}^{|\mathcal I_+|}
\]
such that the corresponding branch problem \reff{mod:Infeas-SLCC-1} is feasible.  If all branch problems \reff{mod:Infeas-SLCC-1} are infeasible, then the original SLCC is infeasible.
\end{proposition}

The proof follows the same arguments as the proof of Proposition~\ref{prop:Infeas} for LCPs, we thus omit the details here.

We remark that Proposition~\ref{prop:Infeas-SLCC} provides a theoretical guarantee for infeasibility detection. In the worst case, this procedure may require solving up to \(2^{|\mathcal I_+|}\) LP feasibility problems. For computational consideration, we apply the recursive disjunctive search only after a relatively small active set \(\mathcal I_+\) has been identified by P-ReCR-SLCC.

Now we describe the detailed Step 2 to check the feasibility of the underlying SLCC.

\par\noindent\rule{\textwidth}{0.4pt}
	\noindent{\bf \hspace*{1.5cm} Step 2.  Final Output and Feasibility Check}
\par\noindent\rule{\textwidth}{0.4pt}
    \begin{itemize}
        \item[]{ Step 2.1} If \((x^{*})^{\!T}\!\bigl(Mx^{*}+Qz^{*}-r\bigr)\le \epsilon\) or \(|\mathcal I_+|=0\), then
        \item[]{} \qquad Output \((x^*,z^*) = (x^k,z^k)\) as a feasible solution to the SLCC;
        \item[]{Step 2.3} Else
         \item[]{}\qquad  \qquad \quad For each \(\delta\in\{0,1\}^{|{\mathcal I}_+|}\)
          \item[]{} \qquad \qquad  \qquad Solve problem \reff{mod:Infeas-SLCC-1} associated with the branch pattern \(\delta\);
           \item[]{} \qquad \qquad \quad If problem \reff{mod:Infeas-SLCC-1} is feasible then
           \item[]{} \qquad \qquad \qquad Output the corresponding feasible solution;
           \item[]{} \qquad \qquad \quad End if
           \item[]{} \qquad \qquad \quad End For
           \item[]{} \qquad \qquad \quad If all the branch problems  \reff{mod:Infeas-SLCC-1} are infeasible then
           \item[]{} \qquad \qquad \qquad The original SLCC is infeasible;
           \item[]{} \qquad \qquad \quad End if
        \item[]{Step 2.4}   End If
    \end{itemize}
\par\noindent\rule{\textwidth}{0.4pt}

\subsection{Rectified Convex Quadratic Relaxation for SLCCs}\label{subsec:ReQR-SLCC}

We next extend ReQR from LCPs to SLCCs. For this, consider the  QO \reff{mod:QO-SLCC}.
Let \(w=(x;z)\). Then the objective in \reff{mod:QO-SLCC} can be written as
\[
\varphi(x,z)=\tfrac12 w^{\!T}H^{+}w-\tfrac12 w^{\!T}H^{-}w-r^{\!T}x,
\]
where \(H^{+},H^{-}\succeq0\) are obtained from the eigenvalue decomposition (EVD) of the matrix
$$
H=\begin{bmatrix}M+M^{\!T}&Q\\ Q^{\!T}&0\end{bmatrix}, \quad H^{+}H^{-}=0.$$
At each iteration we linearize the concave part \(-\tfrac12 w^TH^{-}w\)
around the current iterate and solve the following convex QO
\begin{subequations}\label{QO:DC-SLCC}
\begin{eqnarray}
\min &&\quad \tfrac12 w^{\!T}H^{+}w-(w^{(k-1)})^{\!T}H^{-}w-r^{\!T}x \\
\text{s.t.} &&\quad Mx+Qz\ge r,\quad Ax+Bz=b,\quad x,z\ge0.
\end{eqnarray}
\end{subequations}

Now we are ready to  describe the ReQR for SLCC.
\par\noindent\rule{\textwidth}{0.4pt}
	\noindent{\bf \hspace*{.5cm} Rectified Quadratic Relaxation for SLCC}
\par\noindent\rule{\textwidth}{0.4pt}
\begin{itemize}\itemsep-2pt
	\item[]{\bf Input:} \( M, Q, A, B, r, b\), tolerance \( \epsilon > 0 \), maximum iterations \( K \);
	\item[]{\bf Step 0.} Initialization
	\begin{itemize}
		\item[]{ Step 0.1.} Compute  \( H^{+},  H^{-} \) using the EVD of \( H \);
        \item[]{ Step 0.2.} Set $k=0$;
		\item[]{ Step 0.3.} Solve model \eqref{mod:PLO-SLCC} with \( c_x^{k}=\mathbf{1},  c_z^k=\mathbf{0} \) for \( w^{k} = (x^{k}; z^{k}) \);
	\end{itemize}
		\item[]{\bf Step 1.} Solving process
	\begin{itemize}
        \item[]{ Step 1.1.} Update \( k = k + 1 \);
		\item[]{ Step 1.2.} Solve \ref{QO:DC-SLCC} using $w^{k-1}$; obtain \( w^{k} = (x^{k} , z^{k})\);
		\item[]{ Step 1.3.} If \( \| (x^{k}, z^{k}) - (x^{k-1}, z^{k-1}) \|\geq \epsilon \), go to Step 1.1;
	\end{itemize}
	\item[]{\bf Step 2.} Output \( (x^*, z^*) := (x^{k}, z^{k}) \) and check feasibility.
\end{itemize}
\par\noindent\rule{\textwidth}{0.4pt}


The convergence  of the ReQR-SLCC  follows a similar vein as that  of the standard DC programming framework \cite{Horst1999DC, Thi2011DC-LCP}. Thus, we summarize only the theoretical result and omit the detailed proof here.
\begin{theorem}\label{thm:ReQRSLCC}
Suppose that the QO subproblem at every iteration of ReQR is solved exactly with $\epsilon=0$ and Assumption~\ref{Assump3} holds.  Then   the sequence $\{x^k, z^k\}$ generated by ReQR is  bounded and converges to a stationary point of  problem \reff{mod:QO-SLCC}.
\end{theorem}

As with the ReCR methods for LCPs in Section \ref{Sec:ReCR-LCP}, P-ReCR for SLCCs is a first-order-type method and may struggle to find highly accurate solutions of the associated bilinear reformulation. To address this issue, we adopt the strategy used in H-ReCR for LCPs: switch to ReQR once a good approximate solution of the SLCC has been identified. We call the resulting method hybrid ReCR (H-ReCR) for SLCCs. Since the new H-ReCR for SLCCs follows a similar process as H-ReCR for LCPs, we omit the detailed steps here.

\section{Numerical Experiments}
\label{sec:Numerics}
In this section, we numerically examine the performance of the proposed ReCR approaches, H-ReCR and ReQR, on a variety of LCP and SLCC instances. All computations were carried out using \textsc{Python} on \textsc{Google Colab} with high-memory configurations. The LO/QO subproblems in the ReCR approaches were solved using the \textsc{Gurobi} optimizer under an academic license. The tolerance for all tests was set to \(\epsilon=10^{-6}\), with parameters \(\theta = 0.1\), \(\tau = 10\), and maximum iterations \(K = 101\).

\subsection{Numerical Experiments on LCPs}

In this subsection, we evaluate the performance of H-ReCR and ReQR on three sets of benchmark LCP instances.  The experimental study consists of three parts. First, we test the reliability and robustness  of H-ReCR and ReQR on 10  LCP instances in small scale from the literature, and compare the new ReCR approaches  with several existing LCP solvers, including Lemke’s pivoting algorithm \cite{Lemke1964LCP,Lemke1965}, the path solver \cite{Ferris1999Path} and the path-following method for the non-smooth equation system formulated using  the Fischer--Burmeister (FB) function \cite{Ferris1999MP,Fischer1992}. Particularly, the path solver developed by Ferris and Munson \cite{Ferris1999Path} has been recognized as the most successful solver for complementary problems.
 Second, we test the scalability of ReCR approaches  on large-scale  LCP instances with structured sparse matrix $M$ and simple all-1 vector $r$  from the literature. 
 Third, we test the performance of H-ReCR and ReQR on 4  types of synthetic LCP instances generated in this work.

\subsubsection{Reliability and Robustness Test on Small-scale LCP Instances}

In this subsection, we first list ten small-scale LCP instances collected from the literature and then report the reliability and robustness test results for H-ReCR and ReQR, as well as several popular existing LCP solvers.

\paragraph{LCPs P1--P7 \cite{Floudas1999}.}
These seven LCP instances are from \cite[Chapter~10]{Floudas1999}. For self-completeness, we list all these instances here.
\begin{equation*}
	M_1 =
	\begin{pmatrix}
		1 & 1 \\ 1 & 1
	\end{pmatrix},
	\qquad
	r_1 =
	\begin{pmatrix}
		1 \\ 1
	\end{pmatrix}.
\end{equation*}

\begin{equation*}
	M_2 =
	\begin{pmatrix}
		1 & 2 & 2 & \cdots & 2 \\
		0 & 1 & 2 & \cdots & 2 \\
		\vdots & \vdots & \ddots & \ddots & \vdots \\
		0 & 0 & \cdots & 0 & 1
	\end{pmatrix},
	\qquad
	r_2 = \mathbf{1}_{16}.
\end{equation*}

\begin{equation*}
	M_3 =
	\begin{pmatrix}
		0 & -1 & 2 \\
		2 & 0 & -2 \\
		-1 & 1 & 0
	\end{pmatrix},
	\qquad
	r_3 =
	\begin{pmatrix}
		3 \\ -6 \\ 1
	\end{pmatrix}.
\end{equation*}

\begin{equation*}
	M_4 =
	\begin{pmatrix}
		0 & 0 & 10 & 20 \\
		0 & 0 & 30 & 15 \\
		10 & 20 & 0 & 0 \\
		30 & 15 & 0 & 0
	\end{pmatrix},
	\qquad
	r_4 =
	\begin{pmatrix}
		1 \\ 1 \\ 1 \\ 1
	\end{pmatrix}.
\end{equation*}

\begin{equation*}
	M_5 =
	\begin{pmatrix}
		11 & 0 & 10 & -1 \\
		0 & 11 & 10 & -1 \\
		10 & 10 & 21 & -1 \\
		1 & 1 & 1 & 0
	\end{pmatrix},
	\qquad
	r_5 =
	\begin{pmatrix}
		-50 \\ -50 \\ -10 \\ 6
	\end{pmatrix}.
\end{equation*}

\begin{equation*}
	M_6 =
	\begin{pmatrix}
		11 & 0 & 10 & -1 \\
		0 & 11 & 10 & -1 \\
		10 & 10 & 21 & -1 \\
		1 & 1 & 1 & 0
	\end{pmatrix},
	\qquad
	r_6 =
	\begin{pmatrix}
		-50 \\ -50 \\ -23 \\ 6
	\end{pmatrix}.
\end{equation*}
\begingroup
\setlength{\arraycolsep}{4pt}
\small
\[
M_7 =
\begin{pmatrix}
	0_{6\times 6} &
	\begin{pmatrix}
		1 & 0\\
		1 & 0\\
		1 & 0\\
		0 & 1\\
		0 & 1\\
		0 & 1
	\end{pmatrix} &
	\begin{pmatrix}
		-1 & 0 & 0\\
		0 & -1 & 0\\
		0 & 0 & -1\\
		-1 & 0 & 0\\
		0 & -1 & 0\\
		0 & 0 & -1
	\end{pmatrix}
	\\[10pt]
	\begin{pmatrix}
		-1 & -1 & -1 & 0 & 0 & 0\\
		0 & 0 & 0 & -1 & -1 & -1
	\end{pmatrix} &
	0_{2\times 2} & 0_{2\times 3}
	\\[10pt]
	\begin{pmatrix}
		1 & 0 & 0 & 1 & 0 & 0\\
		0 & 1 & 0 & 0 & 1 & 0\\
		0 & 0 & 1 & 0 & 0 & 1
	\end{pmatrix} &
	0_{3\times 2} & 0_{3\times 3}
\end{pmatrix}
,\qquad
r_7 =
\begin{pmatrix}
	-0.225\\ -0.153\\ -0.162\\ -0.225\\ -0.162\\ -0.126\\ -325\\ -575\\ 325\\ 300\\ 275
\end{pmatrix}.
\]
\endgroup

\paragraph{LCP P8 \cite{ChenYe2000}.}
\[
M_8 =
\begin{pmatrix}
	0 & 0 & 1\\
	0 & 1 & 0\\
	0 & -1 & 1
\end{pmatrix},
\qquad
r_8 =
\begin{pmatrix}
	0\\
	0\\
	-1
\end{pmatrix}.
\]

\paragraph{LCP P9 \cite{Fernandes2001}.}
\[
M_9 =
\begin{pmatrix}
	0 & 0 & -2\\
	0 & 0 & -1\\
	0 & 2 & 1
\end{pmatrix},
\qquad
r_9 =
\begin{pmatrix}
	0\\
	0\\
	-1
\end{pmatrix}.
\]

\paragraph{LCP P10 \cite{Dutta2024LCP}.}
\[
M_{10} =
\begin{pmatrix}
	0 & 0 & 0& 1& 2\\
	0 & 0 & -1 & -1& 2\\
	0 & -1 & 0 & -1 & 1 \\
    1& -1& -1& 0& 0 \\
    2& 1& 0& 0& 0\\
\end{pmatrix},
\qquad
r_{10} =
\begin{pmatrix}
	2\\
	1\\
	-7\\
    -2\\
    1
\end{pmatrix}.
\]

All the above instances are small-scale problems with different structured matrices and vectors, and they are used to test the reliability and robustness of different approaches for LCPs.

The results of the reliability and robustness test are summarized in Table~\ref{tab:section1}, where we list the number of iterations (Iter), and a success indicator (S/F), where ``S'' denotes that the algorithm successfully computed a solution satisfying the complementarity condition \(x^{T}(Mx - r) \leq \epsilon\), and ``F'' indicates failure to meet this criterion.

\begin{table}[htbp]
	\centering
	\caption{Reliability and robustness test on LCP solvers}
	\label{tab:section1}
	\scriptsize
\setlength{\tabcolsep}{3pt}
\renewcommand{\arraystretch}{1.1}
\begin{tabular}{ll*{10}{c}}
	\toprule
	& & \multicolumn{2}{c}{ReQR}
	& \multicolumn{2}{c}{H-ReCR}
	& \multicolumn{2}{c}{PATH}
	& \multicolumn{2}{c}{Lemke}
	& \multicolumn{2}{c}{FB} \\
	\cmidrule(lr){3-4}
	\cmidrule(lr){5-6}
	\cmidrule(lr){7-8}
	\cmidrule(lr){9-10}
	\cmidrule(lr){11-12}
	P & $n$ & Iter & S/F & Iter & S/F & Iter & S/F & Iter & S/F & Iter & S/F \\
	\midrule
	P1  & 2  & 0 & S & 0 & S & 0  & S & 1  & S & 4  & S \\
	P2  & 16 & 0 & S & 0 & S & 0  & S & 31 & S & 11 & F \\
	P3  & 3  & 1 & S & 0 & S & 4  & S & 4  & S & 27 & S \\
	P4  & 4  & 0 & S & 0 & S & 0  & S & 1  & F & 11 & S \\
	P5  & 4  & 1 & S & 3 & S & 5  & S & 4  & S & 24 & S \\
	P6  & 4  & 1 & S & 2 & S & 7  & S & 5  & S & 27 & F \\
	P7  & 11 & 1 & S & 1 & S & 21 & S & 9  & F & 8  & F \\
	P8  & 3  & 0 & S & 0 & S & 0  & S & 0  & S & 1  & S \\
	P9  & 3  & 0 & S & 0 & S & 0  & S & 0  & S & 1  & S \\
	P10 & 5  & 0 & S & 0 & S & 2  & S & 1  & F & 24 & F \\
	\bottomrule
\end{tabular} 
\end{table}

As shown in Table~\ref{tab:section1}, both ReQR and H-ReCR successfully compute approximate feasible solutions for all benchmark instances (P1--P10), consistently satisfying the complementarity conditions within the prescribed tolerance. In many cases, the reported number of iterations is zero, indicating that the solution obtained from the initial convex relaxation already satisfies the complementarity conditions without requiring further refinement. For the remaining instances, only one or a few additional iterations are needed, highlighting the effectiveness of the rectification mechanism in correcting any residual complementarity violation.

In contrast, the classical approaches considered in Table~\ref{tab:section1} exhibit less consistent behavior. While Lemke’s algorithm performs well on several instances, it fails on some problems (e.g., P4, P7 and P10), reflecting its sensitivity to problem structure. Similarly, the Fischer--Burmeister (FB) based approach show multiple failures, particularly on instances indefinite matrices. The PATH solver, although robust in terms of feasibility and successfully solving all instances, generally requires a significantly larger number of iterations, especially on more complex problems such as P5--P7, indicating higher computational effort compared to the proposed approaches. These observations demonstrate clearly that compared with existing approaches for LCPs,  the proposed ReCR approaches are more reliable, robust and scalable.

\subsubsection{Scalability Test on Structured Large-scale LCP Instances}

In this subsection, we first list four large-scale LCP instances collected from the literature and then report the scalability test results for ReQR, H-ReCR and the PATH solver \cite{Ferris1999Path}. The matrices in these four instances are structured and sparse, and $r$ is the all-ones vector. Therefore, we can extend these instances to high-dimensional settings, e.g., \(n = 1000\), \(5000\), \(10000\), \(15000\), \(20000\), \(25000\), \(30000\) and \(40000\), to assess the scalability of the ReCR approaches.

\paragraph{LCP P11 \cite{Ahn1983}.}
\[
M_{11} =
\begin{pmatrix}
	4 & -2 & 0 & \cdots & 0\\
	1 & 4 & -2 & \cdots & 0\\
	0 & 1 & 4 & \ddots & \vdots\\
	\vdots & \ddots & \ddots & \ddots & -2\\
	0 & \cdots & 0 & 1 & 4
\end{pmatrix},
\qquad
r_{11} = e.
\]

\paragraph{LCPs P12--P13 \cite{Geiger1996}.}
\[
M_{12} =
\begin{pmatrix}
	4 & -1 & 0 & \cdots & 0 \\
	-1 & 0 & -1 & \cdots & 0 \\
	0 & -1 & 0 & \ddots & \vdots \\
	\vdots & \ddots & \ddots & \ddots & -1 \\
	0 & 0 & \cdots & -1 & 1
\end{pmatrix},
\qquad
r_{12} = e.
\]

\[
M_{13}
= \mathrm{diag}\!\left(\tfrac{1}{n},\;\tfrac{2}{n},\;\dots,\;1\right),
\qquad
r_{13} = e.
\]

\paragraph{LCP P14 \cite{Murty1988}.}
\[
M_{14} =
\begin{pmatrix}
	1 & 2 & 2 & \cdots & 2 \\
	0 & 1 & 2 & \cdots & 2 \\
	0 & 0 & 1 & \ddots & \vdots \\
	\vdots & \vdots & \ddots & \ddots & 2 \\
	0 & 0 & \cdots & 0 & 1
\end{pmatrix},
\qquad
r_{14} = e.
\]

The results of the scalability test for these four LCP instances are summarized in Table~\ref{tab:section11}. For each method, we report the number of iterations (Iter), the CPU time in seconds, and a success indicator (S/F), where ``S'' denotes that the algorithm successfully computed a solution satisfying the complementarity condition \(x^{T}(Mx - r) \leq \epsilon\), and ``F'' indicates failure to meet this criterion. For larger instances, the entries marked by ``--'' indicate that the PATH solver terminated due to memory or runtime limitations. 

\begin{table}[htbp]
	\centering
	\caption{Scalability test of ReQR and H-ReCR on large-scale LCP instances.}
	\label{tab:section11}
	\scriptsize
\setlength{\tabcolsep}{4pt}
\renewcommand{\arraystretch}{1.1}
\begin{tabular}{ll*{9}{c}}
\toprule
& & \multicolumn{3}{c}{ReQR} & \multicolumn{3}{c}{H-ReCR} & \multicolumn{3}{c}{PATH} \\
\cmidrule(lr){3-5}
\cmidrule(lr){6-8}
\cmidrule(lr){9-11}
P & $n$ & Iter & CPU (s) & S/F & Iter & CPU (s) & S/F & Iter & CPU (s) & S/F \\
\midrule
P11 & 1000  & 0 & 0.0192 & S & 0 & 0.0117 & S & 0 & 0.0174 & S \\
P11 & 5000  & 0 & 0.2136 & S & 0 & 0.2173 & S & 0 & 0.4166 & S \\
P11 & 10000 & 0 & 0.6686 & S & 0 & 0.7497 & S & 0 & 2.5413 & S \\
P11 & 15000 & 0 & 1.4079 & S & 0 & 1.4152 & S & -- & -- & -- \\
P11 & 20000 & 0 & 2.5331 & S & 0 & 2.5427 & S & -- & -- & -- \\
P11 & 25000 & 0 & 3.7836 & S & 0 & 3.8207 & S & -- & -- & -- \\
P11 & 30000 & 0 & 5.3860 & S & 0 & 5.4804 & S & -- & -- & -- \\
P11 & 40000 & 0 & 9.5597 & S & 0 & 9.6625 & S & -- & -- & -- \\
\midrule
P12 & 1000  & 0 & 0.0183 & S & 0 & 0.0187 & S & 0 & 0.0016 & S \\
P12 & 5000  & 0 & 0.2092 & S & 0 & 0.2350 & S & 0 & 0.0071 & S \\
P12 & 10000 & 0 & 0.6734 & S & 0 & 0.6776 & S & 0 & 0.0141 & S \\
P12 & 15000 & 0 & 1.3991 & S & 0 & 1.4204 & S & -- & -- & -- \\
P12 & 20000 & 0 & 2.5440 & S & 0 & 2.5426 & S & -- & -- & -- \\
P12 & 25000 & 0 & 3.7637 & S & 0 & 3.8404 & S & -- & -- & -- \\
P12 & 30000 & 0 & 5.4719 & S & 0 & 5.4457 & S & -- & -- & -- \\
P12 & 40000 & 0 & 9.5454 & S & 0 & 9.6419 & S & -- & -- & -- \\
\midrule
P13 & 1000  & 0 & 0.0177 & S & 0 & 0.0170 & S & 0 & 0.0043 & S \\
P13 & 5000  & 0 & 0.2162 & S & 0 & 0.2107 & S & 0 & 0.0321 & S \\
P13 & 10000 & 0 & 0.6898 & S & 0 & 0.7083 & S & 0 & 1.4029 & S \\
P13 & 15000 & 0 & 1.4183 & S & 0 & 1.4183 & S & -- & -- & -- \\
P13 & 20000 & 0 & 2.6285 & S & 0 & 2.5751 & S & -- & -- & -- \\
P13 & 25000 & 0 & 3.7791 & S & 0 & 3.8312 & S & -- & -- & -- \\
P13 & 30000 & 0 & 5.4631 & S & 0 & 5.4483 & S & -- & -- & -- \\
P13 & 40000 & 0 & 9.6123 & S & 0 & 9.6734 & S & -- & -- & -- \\
\midrule
P14 & 1000  & 0 & 0.0170 & S & 0 & 0.0263 & S & 0 & 0.0018 & S \\
P14 & 5000  & 0 & 0.2096 & S & 0 & 0.2172 & S & 0 & 0.0561 & S \\
P14 & 10000 & 0 & 0.6996 & S & 0 & 0.6829 & S & 0 & 1.7105 & S \\
P14 & 15000 & 0 & 1.4107 & S & 0 & 1.4242 & S & -- & -- & -- \\
P14 & 20000 & 0 & 2.5698 & S & 0 & 2.5732 & S & -- & -- & -- \\
P14 & 25000 & 0 & 3.7668 & S & 0 & 3.8297 & S & -- & -- & -- \\
P14 & 30000 & 0 & 5.3806 & S & 0 & 5.4877 & S & -- & -- & -- \\
P14 & 40000 & 0 & 19.4633 & S & 0 & 9.7303 & S & -- & -- & -- \\
\bottomrule
\end{tabular}

\vspace{2mm}
\footnotesize{-- indicates solver terminated due to memory or runtime limitations (PATH failure).} 
\end{table}


The scalability results reported in Table~\ref{tab:section11} further confirm the efficiency and robustness of the proposed methods on large-scale structured LCP instances (P11--P14). For all tested problem sizes up to $n = 40000$, both ReQR and H-ReCR successfully obtain feasible solutions. A notable feature of these results is that the number of iterations remains equal to zero across all instances. This indicates that, for these structured problems, the initial solution obtained from ReCR approaches already satisfies the complementarity conditions within the required tolerance. In other words, no additional refinement phase is needed, and the algorithms effectively solve the LCP in a single step. For comparison, the PATH solver is able to solve the smaller instances efficiently; however, for larger problem sizes it terminates prematurely due to memory or runtime limitations, as indicated by ``--'' in Table~\ref{tab:section11}.

Overall, the results in Tables~\ref{tab:section1} and~\ref{tab:section11} suggest that the proposed ReCR approaches are reliable, robust and efficient on these benchmark instances. They perform consistently on the small-scale benchmark set and scale effectively to large structured problem sizes with minimal computational effort.

\subsubsection{Test on Synthetic LCPs}

In this subsection, we  generate three different types of synthetic LCP instances and test the performance of H-ReCR, ReQR and PATH on these synthetic LCP instances.

\paragraph{(I) Scaled LCPs with positive diagonal matrix (P15).}
We first generate symmetric LCPs with $M=A^TA$, where both the matrix $A$ and the vector $r$ are generated from random distributions such as uniform or normal distributions. This ensures the resulting LCP has a  convex solution set, denoted by $\mathcal{X}_\mathrm{LCP}$. Then a positive diagonal scaling matrix $D=\diag(d_1,d_2,\cdots,d_n)$ is introduced to make \(DM\) indefinite while preserving the original solution. For example, we can first select some $2\times 2$ submatrix defined by
\[ M_s=\left ( \begin{array}{cc}
	m_{ii} & m_{ij} \\
	m_{ji} & m_{jj}\end{array}\right), \qquad m_{ii}m_{jj} \approx m_{ij}m_{ji}.\]
Then, we carefully select the scaling elements $d_i$ and $d_j$ so that the scaled submatrix is indefinite. By applying such a scaling technique to multiple submatrices, we obtain a scaled matrix $DM$ that is indefinite. The new LCP is constructed by
\[
	x\ge 0, \quad DMx\ge Dr, \quad x^T(DMx-Dr)=0.
\]
From the above construction process, we can see that the new  LCP has the same convex solution set as that of the original LCP.

For each dimension \(n \in \{100, 500, 1000\}\), five instances are generated and the results are averaged. The results are reported in Table~\ref{tab:section2A}.

\begin{table}[htbp]
	\centering
	\caption{Performance of ReQR, H-ReCR and PATH on LCPs (P15)}
	\label{tab:section2A}
	\scriptsize
\setlength{\tabcolsep}{4pt}
\renewcommand{\arraystretch}{1.1}
\begin{tabular}{ll*{9}{c}}
\toprule
& & \multicolumn{3}{c}{ReQR} & \multicolumn{3}{c}{H-ReCR} & \multicolumn{3}{c}{PATH} \\
\cmidrule(lr){3-5}
\cmidrule(lr){6-8}
\cmidrule(lr){9-11}
P & $n$ & Iter & CPU (s) & S/F & Iter & CPU (s) & S/F & Iter & CPU (s) & S/F \\
\midrule
P15 & 100  & 10.2 & 0.88   & S & 12.0 & 0.77   & S & 5  & 0.0067 & S \\
P15 & 500  & 27.6 & 74.65  & S & 14.6 & 42.47  & S & 12 & 0.1768 & S \\
P15 & 1000 & 38.6 & 733.20 & S & 19.8 & 410.96 & S & 11 & 1.1148 & S \\
\bottomrule
\end{tabular} 
\end{table}

The performance results reported in Table~\ref{tab:section2A} summarize the behavior of ReQR, H-ReCR and the PATH solver on feasible random convex LCP instances (P15). All three methods successfully solve the problem for all tested dimensions, indicating that this class of instances is well-posed and numerically stable. A clear trend observed from the table is the growth of computational cost with respect to the problem dimension. Both ReQR and H-ReCR exhibit a significant increase in CPU time as $n$ increases from $100$ to $1000$, reflecting the increased cost of solving the underlying linear and quadratic subproblems. In contrast, the PATH solver remains considerably faster across all instances. One possible explanation is that in both H-ReCR and ReQR, we simply resort to the off-shelf solver to solve the LO/QO subproblems, and did not explore the inherent relationships between the subproblems in two consecutive iterations to speed up the process. In comparison, PATH uses pivot rules or the solution of a linear equation system to update the sequence.

Comparing the two proposed ReCR methods, H-ReCR consistently outperforms ReQR in terms of both iteration count and CPU time. The hybrid strategy effectively reduces the number of iterations required to reach a solution, leading to substantial computational savings, particularly for larger problem sizes. For example, at $n=1000$, H-ReCR reduces the iteration count by nearly half and achieves a significantly lower CPU time compared to ReQR. This demonstrates the advantage of combining global and local refinement strategies in the ReCR framework.

Another important observation is that the number of iterations for both ReQR and H-ReCR remains moderate and grows relatively slowly with the problem size. This suggests that the proposed methods maintain stable convergence behavior even as the dimension increases, despite the noticeable increase in per-iteration computational cost.  Overall, the results indicate that while PATH is highly efficient for this class of convex LCPs, the proposed ReCR approaches, particularly H-ReCR, provide a competitive and robust alternative.

\paragraph{(II) Structured infeasible LCPs (P16).}
We now discuss how to construct infeasible LCP instances that have  a feasible  relaxation model \reff{mod:PLO-LCP} by following a similar structure as in Example \ref{Inf-LCP1}.   For this, we first divide the index set $\mathcal{I}$ into two subsets based on some specific integer  (e.g., $i_0=K<n/2$).
\[ \mathcal{I}_0=\{i\in \mathcal{I}: i\le i_0\}, \quad \bar{\mathcal{I}}_0=\mathcal{I}\setminus\mathcal{I}_0.\]
Then we choose a vector $r$ satisfying the following inequalities:
\begin{eqnarray}  \sum_{i \in \mathcal{I}_0} r_i< 0, \qquad \quad r_i\ge 0, \forall i \in \bar{\mathcal{I}}_0.\label{inq1:Inf-LCP} \end{eqnarray}
   Next we randomly generate a submatrix $M_{0,:}$ consisting of  the first $K$ rows of the matrix $M$ such that the summation of every column in the submatrix is nonnegative, i.e.,
   \begin{eqnarray}\label{inq2:Inf-LCP} \sum_{i=1}^K m_{ij}\ge 0, \forall j\in \mathcal{I}.\end{eqnarray}
       Next we generate $K$ submatrices $M_{k,:}, k=1,\cdots,K$ corresponding to the index sets $\mathcal{I}_k, k=1,\cdots, K$ satisfying
    the following equality
    \[ \sum_{k=1}^K \|\mathcal{I}_k\|_1=n-K.\]
      Each submatrix $M_{k,:}$  will have some hidden structure  similar to that in Example \reff{Inf-LCP1}. For convenience, we can further divide it into two sub-submatrices $M_{k0}$ and $\bar{M}_{k}$ respectively, where
     \[  M_{k0}=(m_{ij}), \quad  i\in \mathcal{I}_k, j\in \mathcal{I}_0, \qquad \bar{M}_{k}=(m_{ij}), \quad i\in \mathcal{I}_k, j\in \bar{\mathcal{I}}_0. \]
      We further require that for every $k\in \{1,\cdots,K\}$, the submatrix $M_{k0}$ has    only one  column with nonzero elements, e.g., all the elements in the $k$-th column of $M_{k0}$   equal 1. Consequentially, we have
      \[ M_{k,:}x=x_k +\bar{M}_k \bar{x}, \quad \bar{x}=(x_{K+1}, x_{K+2},\cdots,x_{K+n}).\]
            Then we choose a submatrix  $\bar{M}^\mathrm{sub}_{k}$ of
      $\bar{M}_{k}$ corresponding to   a subset $\mathcal{I}^\mathrm{sub}_k$ of $\mathcal{I}_k$ such that
      the summation of every column in  $\bar{M}^\mathrm{sub}_{k}$ is nonpositive, i.e.,
      \begin{eqnarray}  \sum_{j\in \bar{\mathcal{I}}_0}  m_{ij}\le 0, \quad \forall i\in \mathcal{I}_k^\mathrm{sub}. \label{inq3:Inf-LCP}                                                                        \end{eqnarray}
Since all  the LCP instances derived from the above process follow a similar structure as Example \reff{Inf-LCP1}, therefore, these LCP instances are all infeasible.

For each dimension \(n \in \{100, 500, 1000\}\), five infeasible instances are generated and the results are averaged.  The results are summarized in Table~\ref{tab:section2B}.

\begin{table}[htbp]
	\centering
	\caption{Performance of ReQR, H-ReCR and PATH on   LCPs (P16)}
	\label{tab:section2B}
	\scriptsize
\setlength{\tabcolsep}{4pt}
\renewcommand{\arraystretch}{1.1}
\begin{tabular}{ll*{9}{c}}
\toprule
& & \multicolumn{3}{c}{ReQR} & \multicolumn{3}{c}{H-ReCR} & \multicolumn{3}{c}{PATH} \\
\cmidrule(lr){3-5}
\cmidrule(lr){6-8}
\cmidrule(lr){9-11}
P & $n$ & Iter & CPU (s) & S/F & Iter & CPU (s) & S/F & Iter & CPU (s) & S/F \\
\midrule
P16 & 100  & 3.6 & 0.15 & S & 8.8  & 0.17 & S & 0 & 0.000061 & U \\
P16 & 500  & 1.2 & 0.39 & S & 9.6  & 1.05 & S & 0 & 0.000104 & U \\
P16 & 1000 & 1.0 & 0.95 & S & 13.4 & 4.43 & S & 0 & 0.000179 & U \\
\bottomrule
\end{tabular} 
\end{table}

The results reported in Table~\ref{tab:section2B} indicate that both ReQR and H-ReCR successfully terminate on all tested problem sizes, demonstrating stable numerical behavior for this class of instances. In contrast, the PATH solver returns an undecided status (U) for all cases, as it did not return an infeasibility certificate under the tested settings. Therefore, while PATH terminates almost instantaneously with zero reported iterations, it does not provide a conclusive feasibility/infeasibility certificate for these instances.

A notable observation is the difference in computational effort among the methods. PATH consistently reports negligible CPU time, reflecting its highly optimized implementation; however, this performance should be interpreted together with the absence of a definitive feasibility/infeasibility conclusion. In contrast, both ReQR and H-ReCR require a moderate number of iterations and higher CPU time, with the computational cost increasing as the problem size grows. This is expected, as the proposed methods rely on solving linear and quadratic subproblems at each iteration.

Comparing the two proposed approaches, ReQR generally requires fewer iterations and less CPU time than H-ReCR for these instances. This suggests that, for the P16 class, the local refinement strategy in ReQR is already effective, and the additional global refinement phase in H-ReCR does not provide further computational advantage. Nevertheless, H-ReCR maintains stable performance and converges reliably across all tested dimensions.

\paragraph{(III) Scaled LCPs with indefinite diagonal matrix (P17).}

In this part we generate feasible symmetric LCPs with positive semidefinite matrices that have convex solution sets. Let $x^*$ be a feasible solution to the symmetric LCP. Next, we carefully assign negative values to some elements of the scaling matrix. For example, if $(Mx^*)_i = r_i$ for some index $i$, we can choose $d_i < 0$ and replace the $i$-th constraint in the original LCP by $d_i(Mx - r)_i \ge 0$. Note that the new LCP constructed in this way may no longer have a convex solution set, as illustrated by the following example:
\[
(M|r)=\left( \begin{array}{cc|c}
	1 & 1 & 1 \\
	1 & 1 & -1
\end{array} \right), \quad x \ge 0, \; Mx - r \ge 0, \; x^T(Mx - r)=0.
\]
The above LCP has a unique solution $(x_1^*, x_2^*) = (1,0)$. We can construct a new LCP as follows:
\[
(M|r)=\left( \begin{array}{cc|c}
	-1 & -1 & -1 \\
	1 & 1 & -1
\end{array} \right), \quad x \ge 0, \; Mx - r \ge 0, \; x^T(Mx - r)=0.
\]
One can easily see that the above LCP has two different solutions: $(1,0)$ and $(0,0)$.

For each dimension \(n \in \{100, 500, 1000\}\), 5 instances are generated using the above procedures, and the reported results correspond to averaged performance. The results are presented in Table~\ref{tab:section3}.

\begin{table}[htbp]
	\centering
	\caption{Performance of ReQR, H-ReCR and PATH on  LCPs (P17) }
	\label{tab:section3}
	\scriptsize
\setlength{\tabcolsep}{4pt}
\renewcommand{\arraystretch}{1.1}
\begin{tabular}{ll*{9}{c}}
\toprule
& & \multicolumn{3}{c}{ReQR} & \multicolumn{3}{c}{H-ReCR} & \multicolumn{3}{c}{PATH} \\
\cmidrule(lr){3-5}
\cmidrule(lr){6-8}
\cmidrule(lr){9-11}
P & $n$ & Iter & CPU (s) & S/F & Iter & CPU (s) & S/F & Iter & CPU (s) & S/F \\
\midrule
P17 & 100  & 25.6 & 1.95    & F & 21.4 & 1.48    & S & 81 & 0.2086 & F \\
P17 & 500  & 50.8 & 166.04  & F & 13.4 & 48.04   & S & 27 & 5.1934 & F \\
P17 & 1000 & 50.0 & 1647.84 & F & 15.2 & 485.27  & S & 39 & 36.7491 & F \\
\bottomrule
\end{tabular} 
\end{table}

From Table~\ref{tab:section3}, we observe that H-ReCR successfully obtains feasible solutions for all the tested instances, whereas both ReQR and the PATH solver fail to do so. This behavior is consistent with the theoretical construction of P17, where the introduction of negative scaling factors may destroy the convexity of the solution set and lead to nonconvex or disconnected feasible regions. In such settings, purely local refinement methods such as ReQR may be trapped in a local solution and fail to locate a global solution. Classical solvers such as PATH, which usually rely on local information at the current iterate, may also struggle to escape the local trap.

In terms of computational performance, the PATH solver remains relatively fast but fails to produce valid solutions across all instances. Similarly,  ReQR requires more iterations and more CPU time, but still fails to find a feasible solution. In contrast, H-ReCR consistently succeeds and requires significantly fewer iterations and lower CPU time compared with ReQR. For example, when $n=500$ and $n=1000$, H-ReCR successfully finds a feasible solution within about one-third of the CPU time used by ReQR.

Overall, the results suggest that the P17 instances represent a challenging subclass of feasible LCPs with a nonconvex solution set. While classical and local methods fail to identify valid solutions on these generated instances, the proposed H-ReCR approach exhibits strong robustness and reliability.

\subsection{Numerical Experiments on SLCCs}

In this subsection, we investigate the reliability and scalability of the proposed H-ReCR-SLCC and ReQR-SLCC algorithms and compare with the PATH solver on several types of synthetically generated SLCC instances. The benchmark problems are constructed using problem structures motivated by the affine variational inequality (AVI) formulations in \cite{CaoFerris1996}, together with additional structured perturbation and infeasibility constructions designed to generate challenging complementarity geometries.

The benchmark includes three types of SLCC instances from the literature, the friction-contact problems, AVIs over compact sets, and the Nash equilibrium problems \cite{CaoFerris1996}, as well as two additional types of SLCC instances generated in this work that are either infeasible or have nonconvex solution set. The benchmark library includes a broad range of structured and challenging SLCC instances, including convex, weakly nonconvex, nonconvex and infeasible instances, which facilitates a comprehensive evaluation of the reliability, scalability and infeasibility-detection capability of the proposed algorithms across diverse  structures.

For each type and each dimension $n \in \{100,500,1000\}$, five representative instances were generated, and the average performance of the tested algorithms is reported.

We next describe the generation of three classes of SLCC instances based on the AVI formulation studied in \cite{CaoFerris1996}. Since these constructions closely follow the framework in \cite{CaoFerris1996}, we summarize only the main features of each class and omit implementation details.

\paragraph{(IV) Friction-contact problems (P18).}
The P18 benchmark problem is derived from the friction-contact AVI formulation, where the contact-force variables are modeled through complementarity conditions and auxiliary variables are introduced through linear coupling constraints. Following the construction in \cite{CaoFerris1996}, we generate sparse banded matrices together with planted feasible solutions to obtain scalable feasible SLCC instances. The number of variables is $m=6+n$and the dimensionality of the matrices is $p=6+n$. The numerical results are summarized in Table~\ref{tab:sectionP18}.

\begin{table}[htbp]
	\centering
	\caption{Performance of ReQR, H-ReCR and PATH on SLCCs (P18)}
	\label{tab:sectionP18}
	\scriptsize
\setlength{\tabcolsep}{4pt}
\renewcommand{\arraystretch}{1.1}

\begin{tabular}{llccccccccc}
\toprule
& & \multicolumn{3}{c}{ReQR} & \multicolumn{3}{c}{H-ReCR} & \multicolumn{3}{c}{PATH} \\
\cmidrule(lr){3-5}
\cmidrule(lr){6-8}
\cmidrule(lr){9-11}

P & $n$ & Iter & CPU (s) & S/F & Iter & CPU (s) & S/F & Iter & CPU (s) & S/F \\
\midrule

P18 & 100  & 0.0   & 0.02360    & S & 0.0  & 0.02280    & S & 2.0   & 0.00495     & S \\
P18 & 500  & 11.0  & 35.50380   & S & 53.0 & 133.54710  & S & 5.0   & 0.22107     & S \\
P18 & 1000 & 101.0 & 6408.82510 & F & 13.0 & 103.93340  & S & 133.0 & 154.34871   & F \\

\bottomrule
\end{tabular} 
\end{table}

As shown in Table~\ref{tab:sectionP18}, all three methods successfully solve the small-scale SLCC instances with $n=100$. In this case, the reported CPU times are very small, and both ReQR and H-ReCR require essentially no additional refinement beyond initialization. The PATH solver also performs efficiently on this instance with only a few iterations and negligible runtime.

For the medium-sized instance with $n=500$, all three methods still succeed, but the computational behavior becomes noticeably different. ReQR requires fewer iterations and less CPU time than H-ReCR on this instance, indicating that the local refinement phase of ReQR remains effective on moderate-scale friction-contact problems. However, PATH substantially outperforms both proposed approaches in terms of CPU time, solving the instance in only a fraction of a second.

For the largest instance with $n=1000$, the behavior of the methods diverges significantly. ReQR fails to find a solution meeting the termination criterion despite taking more than 100 iterations and a much longer computational time. Similarly, the PATH solver also fails on this instance after taking more than 100 iterations. In contrast, H-ReCR successfully computes a feasible SLCC solution within 13 iterations and a shorter CPU time than both ReQR and PATH. These results demonstrate the robustness of H-ReCR on large-scale and more challenging friction-contact-inspired SLCC instances.

Overall, the numerical results on P18 indicate that all methods perform well on small-scale and medium-scale instances, with PATH being the most efficient. However, when the problem size increases to $n=1000$, both ReQR and PATH fail to find a feasible solution to the underlying SLCC, while H-ReCR successfully finds one.

\paragraph{(V) AVIs over compact sets (P19).}
The P19 benchmark problem is derived from compact-set AVI formulations, where bounded box constraints are represented through auxiliary slack variables in SLCC form. Following \cite{CaoFerris1996}, we generate scalable feasible instances using indefinite matrix scaling together with planted feasible solutions satisfying the compact constraints. The dimensions are set as $m=n=p$.  The computational results are summarized in Table~\ref{tab:sectionP19}.

\begin{table}[htbp]
\centering
\caption{Performance of ReQR, H-ReCR and PATH on  SLCCs (P19)}
\label{tab:sectionP19}
\scriptsize
\setlength{\tabcolsep}{4pt}
\renewcommand{\arraystretch}{1.1}

\begin{tabular}{llccccccccc}
\toprule
& & \multicolumn{3}{c}{ReQR} & \multicolumn{3}{c}{H-ReCR} & \multicolumn{3}{c}{PATH} \\
\cmidrule(lr){3-5}
\cmidrule(lr){6-8}
\cmidrule(lr){9-11}

P & $n$ & Iter & CPU (s) & S/F & Iter & CPU (s) & S/F & Iter & CPU (s) & S/F \\
\midrule

P19 & 100  & 5.0  & 0.88610    & S & 7.0  & 0.95150    & S & 6.0  & 0.00546 & S \\
P19 & 500  & 34.0 & 209.44300  & S & 9.0  & 37.22830   & S & 11.0 & 0.31223 & S \\
P19 & 1000 & 49.0 & 2364.67600 & S & 18.0 & 619.23400  & S & 11.0 & 1.62766 & S \\

\bottomrule
\end{tabular} 
\end{table}

As shown in Table~\ref{tab:sectionP19}, all three methods, ReQR, H-ReCR and PATH, successfully solve all SLCC instances in P19. For the small-scale instances with $n=100$, both ReQR and H-ReCR are competitive with PATH in terms of CPU time. However, as the problem size increases, the difference between ReQR and H-ReCR becomes significantly more pronounced. For $n=500$, H-ReCR requires substantially fewer iterations and much less CPU time than ReQR, while PATH remains the most efficient method among the three methods for P19.

This performance gap becomes even more significant for the case $n=1000$. Although all three methods successfully find feasible SLCC solutions, ReQR takes more iterations and uses much longer CPU time than the other two methods. Both H-ReCR and PATH are scalable on these instances, while PATH remains the most efficient method. Overall, the results on P19 show that all three methods are capable of reliably addressing this class of SLCC instances. Among the proposed approaches, H-ReCR consistently exhibits superior scalability and computational efficiency compared with ReQR on medium- and large-scale problems.

\paragraph{(VI) Nash equilibrium problems (P20).}
The P20 benchmark problem is based on Nash equilibrium formulations whose optimality conditions admit an AVI and SLCC representation. Following the structured constructions in \cite{CaoFerris1996}, we generate block-structured positive definite player matrices together with weak inter-player couplings and feasible planted solutions. The number of auxiliary variables  and the number of coupling constraints are defined by $m=n/10$ and $p=n/10$, respectively. The computational results are listed in Table~\ref{tab:sectionP20}.

\begin{table}[htbp]
\centering
\caption{Performance of ReQR, H-ReCR and PATH on SLCCs (P20)}
\label{tab:sectionP20}
\scriptsize
\setlength{\tabcolsep}{4pt}
\renewcommand{\arraystretch}{1.1}

\begin{tabular}{llccccccccc}
\toprule
& & \multicolumn{3}{c}{ReQR} & \multicolumn{3}{c}{H-ReCR} & \multicolumn{3}{c}{PATH} \\
\cmidrule(lr){3-5}
\cmidrule(lr){6-8}
\cmidrule(lr){9-11}

P & $n$ & Iter & CPU (s) & S/F & Iter & CPU (s) & S/F & Iter & CPU (s) & S/F \\
\midrule

P20 & 100  & 1.0 & 0.37660 & S & 3.0 & 0.05770 & S & 6.0 & 0.00130 & S \\
P20 & 500  & 1.0 & 0.56120 & S & 7.0 & 1.09950 & S & 6.0 & 0.00901 & S \\
P20 & 1000 & 1.0 & 2.32930 & S & 7.0 & 5.23780 & S & 8.0 & 0.02674 & S \\

\bottomrule
\end{tabular} 
\end{table}

As shown in Table~\ref{tab:sectionP20}, all three methods, ReQR, H-ReCR and PATH, successfully solve all SLCC instances in P20. It should be pointed out that every instance in P20 has a unique solution due to the positive definite matrix involved in the quadratic objective function for every player. This is consistent with the numerical results in Table~\ref{tab:sectionP20}, as ReQR converges extremely rapidly on all tested instances and requires only one iteration for all the instances within a short CPU time.

The H-ReCR method also successfully solves all tested instances and remains computationally efficient across all problem sizes. Although H-ReCR requires more iterations than ReQR on this benchmark class, the total computational effort remains moderate even for the largest instance with $n=1000$. This indicates that the hybrid refinement framework remains stable and scalable on structured convex game-theoretic SLCCs. Again, PATH remains the most efficient method for P20.

\paragraph{(VII) Scaled  SLCCs with indefinite diagonal matrix (P21).}
We next describe how to generate synthetic SLCC instances that may have a nonconvex solution set and evaluate the performance of the three methods on this class of SLCCs, denoted by P21. The construction of the instances in P21 follows a process similar to the generation of synthetic LCPs P17, where we use a diagonal scaling matrix with some negative diagonal entries to preserve feasibility while potentially producing a nonconvex solution set. This makes this benchmark class more challenging to address.

To construct the benchmark instances, we first choose the problem size $n$ and select the number of auxiliary variables and the number of coupling constraints according to $m=\lceil n/5 \rceil$ and $p = \lceil n/10 \rceil$. Next we follow the construction process in \cite{CaoFerris1996} to  generate an SLCC instance with a symmetric positive definite matrix $M$ and a known unique solution denoted by $(x^*, z^*)$. Let
\[
\mathcal{I}^*_{+} = \{i \in \mathcal{I}: x^*_i > 0\}.
\]
We next select a subset $\mathcal{I}^1_+$ of the above index set and use the selected subset to construct a diagonal matrix
\[
D=\operatorname{diag}(d_1,\dots,d_n), \qquad \text{where } d_i=
\begin{cases}
-1, & i\in \mathcal{I}^1_{+},\\
1, & \text{otherwise}.
\end{cases}
\]
We then replace the complementary constraints in the original SLCC by the following new constraints:
\[
 D(Mx+Qz) \ge Dr, \qquad x^T D(Mx+Qz-r)=0.
\]
The planted point $(x^*,z^*)$ remains feasible for the scaled SLCC, so the new SLCC is feasible. Like the LCP instances in P17, the solution set of the new SLCC may be nonconvex. The numerical results for P21 are summarized in Table~\ref{tab:sectionP21}.

\begin{table}[htbp]
\centering
\caption{Performance of ReQR, H-ReCR and PATH on SLCCs (P21)}
\label{tab:sectionP21}
\scriptsize
\setlength{\tabcolsep}{4pt}
\renewcommand{\arraystretch}{1.1}

\begin{tabular}{llccccccccc}
\toprule
& & \multicolumn{3}{c}{ReQR} & \multicolumn{3}{c}{H-ReCR} & \multicolumn{3}{c}{PATH} \\
\cmidrule(lr){3-5}
\cmidrule(lr){6-8}
\cmidrule(lr){9-11}

P & $n$ & Iter & CPU (s) & S/F & Iter & CPU (s) & S/F & Iter & CPU (s) & S/F \\
\midrule

P21 & 100  & 32.0  & 5.63118     & F & 4.0  & 0.12008    & S & 109 & 0.28518    & F \\
P21 & 500  & 101.0 & 870.02632   & F & 14.0 & 74.88307   & S & 114 & 12.23437   & F \\
P21 & 1000 & 101.0 & 6408.82510  & F & 13.0 & 604.38266  & S & 149 & 107.60681  & F \\

\bottomrule
\end{tabular} 
\end{table}

As shown in Table~\ref{tab:sectionP21}, the SLCC instances in P21 are more challenging to solve. Both ReQR and PATH fail on all tested instances in P21, and H-ReCR is the only tested method able to find a feasible solution to all these instances. This illustrates the robustness of H-ReCR on this class of synthetic SLCC instances.

\paragraph{(VIII) Infeasible SLCC instances (P22).}
We next describe how to generate synthetic SLCCs that are infeasible but with a feasible relaxation model, and evaluate the performance of three methods on these infeasible SLCC instances. These benchmark problems are constructed to evaluate the capability of the proposed methods to detect the infeasibility of the underlying SLCC for which no feasible  solutions exist. The construction of these infeasible SLCC instances follows a similar process like the infeasible LCP instances (P16). For self-completeness, we describe the process below.

To generate these infeasible SLCC instances, we first choose the size $n$, and  select the number of auxiliary variables  and the number of coupling constraints according to $m=\lceil n/5 \rceil$ and $p=\lceil n/10 \rceil$. Next, we generate a random feasible SLCC structure before applying a structured infeasibility perturbation. The matrices $M_0 \in \mathbb{R}^{n\times n}$, $Q_0 \in \mathbb{R}^{n\times m}$, $A \in \mathbb{R}^{p\times n}$, and $B \in \mathbb{R}^{p\times m}$ are generated independently with entries sampled from the standard normal distribution $\mathcal{N}(0,1)$.
A feasible solution pair $(x^{*},z^{*})$ is then generated using randomized sparsity patterns. Specifically, each component of $x^{*}$ is independently activated with probability $\rho_x = 0.5$, and the active entries are sampled independently from the uniform distribution on $[0.1,1.0]$. The auxiliary vector $z^* \in \mathbb{R}^m$ is generated with entries sampled independently from the uniform distribution on $[0,1]$.  Let $\mathcal{I}_{0} = \{i : x_i^{*}>0\}$ denote the support set of the generated complementarity variables. A subset $\mathcal{I}_{1}\subseteq \mathcal{I}_{0}$ is then selected randomly. These indices are subsequently used to construct conflicting complementarity conditions that destroy feasibility of the SLCC system.

To construct the linear coupling constraints, the first row of the matrix $A$ is modified so that
\[
A_{1j}=
\begin{cases}
1, & j\in \mathcal{I}_{1},\\
0, & \text{otherwise},
\end{cases}
\]
while the first row of $B$ is set to zero. The equality vector is then defined by $b = Ax^{*}+Bz^{*}$, which guarantees consistency with the originally generated feasible point prior to the infeasibility perturbation. A nonnegative slack vector $w_0 \in \mathbb{R}^n$ is then constructed such that $(w_0)_i=0$ whenever $x_i^{*}>0$, and otherwise $(w_0)_i = 0.1+\xi_i$ with $\xi_i \sim \mathrm{Unif}[0,1]$.  The preliminary right-hand side vector is then defined by $r_0=M_0x^{*}+Q_0z^{*}-w_0$.  To enforce infeasibility, we perturb the rows associated with the index set $\mathcal{I}_{1}$ by replacing the corresponding rows of $M_0$ and $Q_0$ with strictly positive coefficients. More precisely, for every $i\in\mathcal{I}_{1}$, we set
\[
(M_0)_{i,:} = |(M_0)_{i,:}|+1, \qquad (Q_0)_{i,:} = |(Q_0)_{i,:}|+1,
\]
and simultaneously assign $(r_0)_i=-1$. This perturbation generates a conflicting system in which the complementarity conditions cannot be simultaneously satisfied while preserving nonnegativity of the variables. Consequently, the resulting SLCC instance becomes infeasible.

Finally, to further diversify the problem geometry, we apply a positive diagonal scaling transformation. Specifically, we generate a diagonal matrix $D=\operatorname{diag}(d_1,\dots,d_n)$, where each scaling coefficient is sampled independently from the uniform distribution on $[0.2,3.0]$. The final SLCC matrices are then defined by
\[
	M = DM_0, \qquad Q = DQ_0, \qquad r = Dr_0.
\]
The resulting SLCC instances form a challenging class of infeasible SLCC instances with dense matrices and deliberately conflicting complementarity constraints. These problems are useful for evaluating the robustness of the proposed methods in identifying infeasibility and handling difficult complementarity structures in large-scale SLCC formulations.
The average computational results is reported in Table~\ref{tab:sectionP22}.

\begin{table}[htbp]
\centering
\caption{Performance of ReQR, H-ReCR and PATH on  SLCCs (P22)}
\label{tab:sectionP22}
\scriptsize
\setlength{\tabcolsep}{4pt}
\renewcommand{\arraystretch}{1.1}

\begin{tabular}{llccccccccc}
\toprule
& & \multicolumn{3}{c}{ReQR} & \multicolumn{3}{c}{H-ReCR} & \multicolumn{3}{c}{PATH} \\
\cmidrule(lr){3-5}
\cmidrule(lr){6-8}
\cmidrule(lr){9-11}

P & $n$ & Iter & CPU (s) & S/F & Iter & CPU (s) & S/F & Iter & CPU (s) & S/F \\
\midrule

P22 & 100  & 29.0  & 3.89049     & S & 57.0  & 3.18922     & S & 97  & 0.62272    & F \\
P22 & 500  & 101.0 & 458.23897   & S & 101.0 & 437.68681   & S & 119 & 18.45452   & F \\
P22 & 1000 & 101.0 & 3413.67647  & S & 101.0 & 3206.27877  & S & 95  & 204.65930  & F \\

\bottomrule
\end{tabular} 
\end{table}

As shown in Table~\ref{tab:sectionP22}, both ReQR and H-ReCR successfully detect the infeasibility of all the SLCC instances in P22.  This shows that the proposed procedure for feasibility check is capable of reliably detecting the infeasibility of  these  infeasible SLCC instances. In contrast, the PATH solver did not return an infeasibility certificate under the tested settings.

With a close look, we find that for SLCC instances with $n=100$, ReQR uses fewer iterations and slightly longer CPU time to detect infeasibility than H-ReCR. However, as the problem size increases, both methods require intensive computational effort to detect infeasibility. For instances with $n=500$ or $n=1000$, the number of iterations reaches the maximum number of iterations used in ReCR, and the corresponding CPU time also becomes much longer due to the process of obtaining an infeasibility certificate. Nevertheless, both methods still successfully generate valid infeasibility certificates, demonstrating strong robustness in handling infeasible SLCCs. These results highlight a distinction between the proposed methods and classical solvers for SLCCs: the proposed ReCR framework can provide an infeasibility certificate when the underlying SLCC instance is infeasible.

\section{Conclusions} \label{sec:conclusion}

In this paper, we introduced a relaxation-rectification framework (ReCR) for SLCCs based on the new universal relaxation theory that allows us to recast the problem of determining the  feasibility of an SLCC as an equivalent bilinear optimization problem. We explored the resulting bilinear model and derived necessary and sufficient optimality conditions for its global solutions. Using these theoretical insights, we developed several ReCR approaches aimed at finding solutions of the associated bilinear reformulation. We established convergence properties of the proposed ReCR approaches under stated assumptions and evaluated their performance on SLCC instances from the literature and on synthetic instances. The theoretical analysis and numerical results suggest that the proposed ReCR approaches are reliable, robust, and scalable on the tested instances and can handle both feasible and infeasible SLCCs.

The ReCR framework provides a reliable approach for SLCCs and also suggests new directions for other classes of structured nonconvex optimization problems. Several issues regarding the framework require further investigation. One direction is to improve the computational efficiency of the ReCR methods. In principle, both P-ReCR and PD-ReCR can be viewed as first-order-type methods, as only an LO subproblem needs to be solved at each iteration. Since these LO subproblems share the same constraint set and differ only in the objective function, it would be useful to investigate warm-starting techniques and related strategies for accelerating the solution process. Moreover, from a theoretical perspective, it would be useful to investigate whether an infeasibility certificate can be derived directly when ReCR terminates at a point that does not satisfy the SLCC complementarity conditions.

Another direction is to extend ReCR approaches to other systems involving linear constraints and nonconvex restrictions, such as systems with binary variables and linear constraints (SBLCs) and systems with linear and conic complementarity constraints (SL3Cs). The SLCC model considered in this paper includes SBLCs as a special case. However, in the SBLC case, the relevant matrix is $M=-I_n$, so the corresponding QO reformulation involves minimizing a concave function over a polyhedral set, which can lead to many local minimizers. This structure can affect the performance of ReCR methods for SBLCs, as observed in our preliminary study \cite{ReCR-SBLC2025}. Addressing this issue requires a problem-specific relaxation theory for SBLCs and rectification functions tailored to their structure. Consequently, new theoretical analysis is needed to study the convergence behavior of the resulting ReCR methods. We leave a detailed investigation of these directions to future works \cite{ReCR-SBLC2025,ReCR-SL3C2025}.

Finally, our broader goal is to develop trustworthy and explainable approaches for structured nonconvex optimization problems, and this work represents an initial step in that direction. The SL3C model provides a unified framework for optimization problems in which the objective and constraint functions are linear or quadratic. It also covers certain nonlinear optimization problems, such as $k$-means clustering. For certain minimization problems, the first ReCR iteration may provide a lower bound on the optimal value of the original problem. On the other hand, when ReCR also produces a feasible solution, its objective value provides an upper bound. These bounds can then be combined with classical search techniques to compute approximate solutions of the original optimization problem. Another possibility is to reformulate the underlying problem as an equivalent SL3C and design ReCR methods for the resulting system; Encouraging preliminary progress in this direction has been reported in \cite{ReCR-QCQO2025}. It may also be possible to develop ReCR approaches for polynomial optimization based on semidefinite relaxations. Further study is needed to explore these possibilities.

\noindent{Acknowledgement:} In the preparation of this work, we have discussed with numerous experts in the optimization community including Prof. Gorge Lan, prof. Arkadi Nemirovskii and Prof. Nick Sahinidis from GaTech, Prof. Yinyu Ye from Stanford University, prof. Jongshi Pang and Prof. Suvrajeet Sen from USC. Their comments and suggestions are highly appreciated.

\end{document}